\documentclass[reqno]{amsart}
\usepackage[all]{xy}
\usepackage{color}
\usepackage{xcolor}
\usepackage{mathrsfs}
\usepackage{hyperref}
\usepackage{enumerate}
\usepackage{amsmath}
\usepackage{amsthm}
\usepackage{amssymb}
\usepackage{amscd}
\usepackage{graphicx}
\usepackage{epsfig}
\usepackage[english]{babel}

  \newcommand{\w}{\omega}
  
\newcommand{\R}{\mathbb{R}}
\newcommand{\C}{\mathbb{C}}

\newcommand{\T}{\mathbb{T}}
\newcommand{\Z}{\mathbb{Z}}

     \newcommand{\cC}{\mathcal{C}}
   \newcommand{\cD}{\mathcal{D}}
      
   \newcommand{\cF}{\mathcal{F}}

     \newcommand{\cI}{\mathcal{I}}
    \newcommand{\cL}{\mathcal{L}}

  \newcommand{\cP}{\mathcal{P}}

\renewcommand{\tt}{\mathfrak{t}}        
  
\renewcommand{\ss}{\mathfrak{s}}
\newcommand{\Gl}{\mathrm{Gl}(n)^{+}}

\DeclareMathOperator{\Ad}{Ad}           

\allowdisplaybreaks

\newtheorem{theorem}{Theorem}[subsection]
\newtheorem{lemma}[theorem]{Lemma}
\newtheorem{proposition}[theorem]{Proposition}

\theoremstyle{definition}
\newtheorem{definition}[theorem]{Definition}
\newtheorem{example}[theorem]{Example}

\newtheorem{remark}[theorem]{Remark}

\begin{document}

\begin{abstract}
We discuss a  system of third order PDEs for strictly convex smooth functions on domains of Euclidean space.
We argue that it may be understood as a closure of sorts of the first order prolongation of a family of second order PDEs.
We describe explicitly its real analytic solutions and
all the solutions which satisfy a genericity condition; we also describe a family of non-generic solutions
which has an application to Poisson geometry and Kahler structures on toric varieties.
Our  methods are geometric: we use the theory of Hessian metrics and symmetric spaces to link the analysis of the system of PDEs with properties of the manifold of matrices with
orthogonal columns.
\end{abstract}

\title[PDEs from matrices with orthogonal columns]{Partial differential equations from matrices with orthogonal columns}

 \author{David Mart\'inez Torres}
%

 \maketitle

\section{Introduction}

Let $\phi$ be a strictly convex smooth function defined on a connected subset $\Omega\subset \R^{n}$. Its Hessian $H\phi$ defines at each
point an inner product and therefore
the inverse of the Hessian matrix is also a smooth field of inner products. 
It is natural to ask whether this field is also the Hessian of a function.
If $g$ is a field of inner products which is the Hessian of a function, then  there must be an equality of partial derivatives: 
\[\frac{\partial g_{ij}}{\partial x_k}=\frac{\partial g_{ik}}{\partial x_j},\quad 1\leq i,j,k \leq n.\]
If $\Omega$ has trivial first homology group, then the agreement of the above partial derivatives implies that $g$ is the Hessian
of a function \cite{Du}. 
We shall assume from now on that the domain $\Omega$ as trivial first homology group. 

\begin{definition}\label{def:inversible} A strictly convex function $\phi\in C^{\infty}(\Omega)$ has {\bf property $\mathcal{I}$}  if it satisfies
the system of third order PDEs:
\begin{equation}\label{eq:inversible}
 \frac{\partial}{\partial x_k}{H\phi^{-1}}_{ij}-\frac{\partial}{\partial x_j}{H\phi^{-1}}_{ik}=0,\quad 1\leq i,j,k \leq n.
\end{equation}
\end{definition}


The purpose of this paper is to analyze the system of third order PDEs (\ref{eq:inversible}) for strictly convex functions.

To be more precise about our focus,  we
note that  it is possible to construct strictly convex solutions to (\ref{eq:inversible}) by elementary means:
Every strictly convex function of one variable has property $\mathcal{I}$. If $\phi_1(x_1)$ and $\phi_2(x_2)$ are strictly convex functions of one variable,
then $\phi_1(x_1)+\phi_2(x_2)$ has  property $\mathcal{I}$. 
Such a function solves the second order hyperbolic PDE with constant coefficients
\begin{equation}\label{eq:2dim-axis-characteristics}
\frac{\partial^{2} \phi}{\partial x_1\partial x_2}=0,
\end{equation}
and, conversely, all the solutions of (\ref{eq:2dim-axis-characteristics})  decompose (locally) as the sum of two functions on each of the variables $x_1$ 
and $x_2$; the parallel translates 
of the coordinate  axis are the (constant) characteristics
of the solutions. If  (\ref{eq:2dim-axis-characteristics}) is replaced by any second order hyperbolic PDE 
with constant coefficients whose solutions have (constant)  orthogonal characteristics, then its strictly convex solutions will 
have property $\mathcal{I}$. 
There is a natural generalization of this family of hyperbolic second order PDEs to arbitrary dimensions. Its strictly convex solutions, which we refer to as 
functions with (constant) orthogonal characteristics, will also have property $\mathcal{I}$.

It is thus natural to study `how close'  a strictly convex function with property $\mathcal{I}$ may be from having orthogonal characteristics. 

Our main results describe sufficient conditions for a strictly convex function with property $\mathcal{I}$ to have orthogonal characteristics. Among such sufficient
conditions there is a generic one:
\begin{theorem}\label{thm:propertyI-generic} If a strictly convex function has property $\mathcal{I}$
and at every point the eigenvalues of its Hessian
 are simple, then it has orthogonal characteristics.
\end{theorem}
Another sufficient condition refers to the regularity of the functions:
\begin{theorem}\label{thm:propertyI-analytic} If a real analytic strictly convex function has property $\mathcal{I}$, then it has orthogonal characteristics. 
\end{theorem}
We shall prove  Theorem \ref{thm:propertyI-generic} first for functions of two variables. For them the system (\ref{eq:inversible}) has two equations and 
the theorem will follow from an algebraic manipulation valid under the hypothesis on the Hessian.
The algebraic manipulation will have a geometric counterpart:  
The family of second order hyperbolic PDEs with constant coefficients whose solutions have orthogonal characteristics 
defines a pencil of hyperplanes on the space of jets of order two; away from 
its base, and in the set where the Hessian is strictly positive, it restricts to a foliation. Strictly convex functions with property
$\mathcal{I}$ define a subset\footnote{That strictly convex functions  with orthogonal characteristics have property 
$\mathcal{I}$ means that this subset contains the prolongation of any of the previous hyperplanes.} of the space of jets of order three.
The hypothesis on the eigenvalues of
the Hessian singles out the locus of smooth points for which the jet projection is a submersion; its image is the aforementioned foliated open
subset of the jets of order two. The geometric 
manifestation of our
algebraic manipulation will be that the Cartan connection is tangent to the leaves of the pullback foliation. This is why one may say that for strictly
convex functions the system of third order PDEs (\ref{eq:inversible}) is the closure of the prolongation of the aforementioned pencil of second order hyperbolic PDEs.  

To go to arbitrary dimensions we will not follow the jet space approach, as we find the algebraic complexities difficult to manage. We shall
switch our viewpoint to that of Hessian metrics. In this language what we are asking is when for a given Hessian metric on a domain of Euclidean space 
its inverse metric is also Hessian. We refer to such metrics as Hessian metrics with property $\mathcal{I}$. The first manifestation of the relevance of the
metric viewpoint will be  the following:

\begin{lemma}\label{lem:propertyI-equals-symmetric} A Hessian metric  has property $\mathcal{I}$ if and only if its Christoffel symbols 
 (of the second kind) are symmetric in the three indices. 
\end{lemma}
Lemma \ref{lem:propertyI-equals-symmetric} suggests that property $\mathcal{I}$ could be described as a feature of the tangent or
the orthogonal frame bundle of the Hessian metric with its Levi-Civita connection. A fundamental property of 
Hessian metrics on domains of Euclidean space is that they posses a universal (positively oriented) orthogonal frame bundle with connection \cite{Du},
$\pi:(\Gl,\nabla) \to \cP$, where $\cP$ denotes positive matrices and $\pi$ and $\nabla$ are a natural map and connection, respectively,  which come from symmetric space theory. 
We shall argue that the universal orthogonal frame bundle offers an appropriate replacement for the jet space picture.
Briefly, jet spaces of order two will be replaced by  positive matrices $\mathcal{P}$;
the subset of the jet spaces of order three defined by property $\mathcal{I}$ will be replaced by the 
submanifold of matrices with orthogonal columns
 $\mathcal{C}\subset \Gl$; the restriction of the jet projection 
will correspond to  $\pi|_{\mathcal{C}}: \mathcal{C}\to \mathcal{P}$; the Cartan connection will correspond to 
 the universal Levi-Civita connection $\nabla$.
Property $\mathcal{I}$ for a Hessian metric will translate as follows:
\begin{theorem}\label{thm:differential-relation}
A Hessian metric $H\phi$ in $\Omega$ has property $\mathcal{I}$ if and only if for any point $x\in \Omega$ and any curve 
$\gamma$ at $x\in \Omega$ there exist an orthonormal frame for $H\phi(x)$ in $\mathcal{C}$ such that the corresponding
horizontal lift of $\gamma$ at that frame  is tangent to $\cC$.
\end{theorem}
 There will be a property analogous to the tangency of the Cartan connection to the pullback foliation coming from the pencil of degree two hyperbolic PDEs: 
\begin{proposition}\label{pro:Cartan} The restriction of the universal Levi-Civita connection to $\mathcal{C}$ defines a (regular) involutive distribution.
Its leaves are the left translates of the strictly positive matrices $\mathcal{D}$ which fit in the Cartan factorization
$\mathcal{C}=\mathrm{SO}(n)\mathcal{D}$.
\end{proposition}
The generic condition on eigenvalues in Theorem \ref{thm:propertyI-generic} is just the open stratum of a natural stratification of $\mathcal{P}$.
To describe more precise sufficient conditions for a Hessian metric with property $\mathcal{I}$ to come from a function with orthogonal characteristics, 
we will analyze the interaction among this stratification, the foliation on
$\mathcal{C}$ defined by the universal Levi-Civita connection, and the map $\pi|_{\mathcal{C}}$. Theorem \ref{thm:propertyI-analytic} will hinge on
real analytic features of these objects.

As we shall see, property $\mathcal{I}$ for strictly convex functions appears  in a problem of Poisson geometry in toric varieties.
The so-called totally real toric  Poisson structures have properties analogous to that of Hamiltonian Kahler forms. For instance, whereas the latter 
are encoded by appropriate strictly convex functions \cite{Gui}, the former are encoded by the simplest strictly convex functions: quadratic forms.
%
%
The most natural Poisson-theoretic PDE for a pair given by a totally real toric  Poisson structure and a Hamiltonian Kahler form will correspond 
to 
property $\mathcal{I}$:

\begin{theorem}\label{thm:propertyI-Poisson} Let $(X,\mathbb{T})$ be a (smooth) toric variety endowed with a totally real toric Poisson structure $\Pi$ 
and a Kahler form $\sigma$ 
for with the action of the maximal compact torus
$T\subset \T$ is Hamiltonian. Let $P$ denote the inverse Poisson structure to $\sigma$.  Then the following statements are equivalent:
\begin{enumerate}
 \item The Poisson structures $\Pi$ and $P$ Poisson commute: $[\Pi,P]=0$.
 \item In a basis of the Lie algebra of $T$ for which $\Pi$ corresponds to the standard quadratic from of $\R^{n}$,
 the strictly convex function which corresponds to $\sigma$ has property $\mathcal{I}$. 
\end{enumerate}

\end{theorem}

In complex dimension one a totally real toric Poisson structure and (the inverse of) a Hamiltonian Kahler form always Poisson commute
because the commutator is a field of trivectors on a surface; equivalently, if we use
Theorem \ref{thm:propertyI-Poisson} this corresponds to the fact that all strictly 
convex functions of one variable have property $\mathcal{I}$. As it will turn out Theorem \ref{thm:propertyI-analytic} will imply 
that in the real analytic category such a commuting pair is the Cartesian product of one dimensional commuting pairs
:

\begin{theorem}\label{thm:no-analytic-intro} Let $(X,\T)$ be a projective toric Poisson variety endowed with a
 totally real toric Poisson structure $\Pi$ which Poisson commutes with a real analytic Hamiltonian Kahler structure $\sigma$. Then  $(X,\T)$ is a Cartesian product of projective lines and both 
$\Pi$ and $\sigma$ factorize.
\end{theorem}

We shall also describe a family of strictly convex functions which satisfy property $\mathcal{I}$ but which do not have orthogonal characteristics. 
We will use it to construct Hamiltonian Kahler forms in certain ($T$-invariant) regions of toric varieties. These regions can be thought of as the result
of gluing to a ($T$-round)  0-handle several ($T$-round) 1-handles. An illustration of the construction is the following:

\begin{proposition}\label{pro:cp2} Let $U\subset \mathbb{C}P^{2}$ be the complement of small $T^{2}$-invariant neighborhoods of $[1:0:0]$ and $[1:0:1]$. Then there
exist a totally real toric Poisson structure on $\mathbb{C}P^{2}$ and a Hamiltonian Kahler form on $U$ which  
Poisson commute.
\end{proposition}

The structure of this paper is as follows. Section \ref{sec:constant-characteristics} describes how matrices with orthogonal columns are used 
to define the family of differential relations of second order with constant coefficients
whose solutions we call functions with (constant) orthogonal characteristics; we also discuss why they have property $\mathcal{I}$.
In Section \ref{sec:2dimensions} we do the analysis of the system of third order PDEs (\ref{eq:inversible}) for strictly convex functions of two variables using jet spaces.
The viewpoint of Hessian metrics is introduced in Section \ref{sec:Hessian}. Property $\mathcal{I}$ is translated
as symmetry of the Christoffel symbols, an algebraically simpler condition which allows to analyze the interaction 
of property $\mathcal{I}$ with the Legendre transform. Section \ref{sec:frame-bundle}  describes how the  universal orthogonal frame bundle offers
the appropriate setting for the geometric analysts of property $\mathcal{I}$. We analyze the map $\pi:(\cC,\nabla)\to \cP$ from the submanifold of orthogonal
matrices with the restriction of the universal Levi-Civita connection onto the manifold of positive matrices; this is the our replacement 
of the subsets defined by property $\mathcal{I}$ in the space of jets of order three and two with the restriction of the Cartan connection.
Section \ref{sec:diferential-relation} contains our main results: Firstly, the description of property $\cI$ as a differential relation 
related to the submanifold of orthogonal matrices $\mathcal{C}\subset (\Gl,\nabla)\to \mathcal{P}$. Secondly,  sufficient conditions for a Hessian metric with 
property $\cI$ to come from a strictly convex function with orthogonal characteristics. In Section \ref{sec:bifurcations} we describe
a family of strictly convex functions which have property $\cI$ but do not have in general orthogonal characteristics. The domains
of definition of its members are what we call polytopes with 1-handles. Despite polytopes with 1-handles not being convex in general, we show 
that the family is invariant under Legendre transform. Section \ref{sec:Poisson-commuting} contains our applications to Poisson geometry. We explain
how on a smooth toric variety the Poisson commuting equation for a totally real toric Poisson structure and for (the inverse of) a Hamiltonian Kahler form can be 
rewritten as property $\cI$ for either the Kahler or the symplectic potential \cite{Gui} of the latter form. That 
allows us to conclude that in the real analytic
category any such commuting pair  must be the Cartesian product of commuting pairs on projective lines. We also use 
the family introduced in Section \ref{sec:bifurcations} to construct commuting pairs
on certain topologically non-trivial regions of toric varieties (which are not Cartesian products). 
 

The author is grateful to Max Planck Institute for Mathematics in Bonn for its hospitality and finantial support.

\section{Solutions with orthogonal characteristics}\label{sec:constant-characteristics}

It is possible to construct strictly convex functions with property $\mathcal{I}$ by elementary means. 
\begin{enumerate}[(a)]
 \item Every strictly convex function of one variable has property $\mathcal{I}$.
 \item If $\phi_1(x_1),\dots,\phi_n(x_n)$ are strictly convex, then 
$\phi_1(x_1)+\cdots+\phi_n(x_n)$ has property $\mathcal{I}$ in the product of the corresponding intervals.
\item  If $\phi(x)$ has property $\mathcal{I}$ in $\Omega$ and $B\in \mathrm{O}(n)$ is an orthogonal transformation, then $\phi(Bx)$ has property $\mathcal{I}$ 
in $B^{-1}(\Omega)$. This is because 
\[H\phi(Bx)=B^{T}H\phi(x)B\] and, therefore, if $H\phi(x)^{-1}$ is the Hessian of $\psi(x)$, then $(H\phi(Bx))^{-1}$ is the Hessian of 
$\psi(Bx)$.
\end{enumerate}

A function is (locally) of the form  $\phi=\phi_1(x_1)+\cdots+\phi_n(x_n)$ if and only if it is a solution of the  system of second order PDEs
\begin{equation}\label{eq:axis-characteristics}
\frac{\partial^{2} \phi}{\partial x_i\partial x_j}=0,\quad 1\leq i<j\leq n.
\end{equation}
The solutions of the system have (constant) characteristics given by the collection of axis.
This information can be used to rewrite (\ref{eq:axis-characteristics})
in a more geometric fashion. For any $n\times n$ matrix $A$ one can define a differential operator of order two on functions with values on matrix-valued functions
by the 
following recipe: 
\begin{equation}\label{eq:constant-characteristics}
\mathcal{L}^{2}_A\phi:=A^{T}H\phi A. 
\end{equation}
Equivalently, the $ij$-th component is the Lie derivative of $\phi$ with respect to the (constant) vector field defined by the $i$-th column of $A$,
followed by the Lie derivative with respect to the vector field defined by the $j$-th column.

Let $\mathcal{D}$ denote the set of diagonal matrices with strictly positive entries and let $\mathcal{D}(\Omega)$ denote smooth functions 
on $\Omega$ with values on $\mathcal{D}$. It follows that a function $\phi$ is strictly convex and satisfies (\ref{eq:axis-characteristics})
if and only if 
$\mathcal{L}^{2}_\mathrm{I}\phi\in \mathcal{D}(\Omega)$,
where $\mathrm{I}$ is the identity matrix. 
Let $\mathcal{C}$ denote the set of matrices with orthogonal columns.
\begin{definition}\label{def:orthogonal-characteristics} A function $\phi\in C^{\infty}(\Omega)$ has {\bf orthogonal characteristics} if there 
exists  $C\in \mathcal{C}$ such that
 \begin{equation}\label{eq:orthogonal-characteristics}
 \mathcal{L}^{2}_C\phi\in \mathcal{D}(\Omega).
\end{equation}
 \end{definition}
Because $\mathcal{D}$ is invariant by conjugation by permutation matrices, in Definition \ref{def:orthogonal-characteristics} we may assume
that $C$ has positive determinant. We will abuse notation and use  $\mathcal{C}$ to refer to 
matrices with orthogonal columns and positive determinant. 

\begin{lemma}\label{lem:orthogonal-characteristics} If $\phi\in C^{\infty}(\Omega)$ has orthogonal characteristics,
 then $\phi$ is an strictly convex function with property $\mathcal{I}$.
More precisely, $\phi$ is the composition of an orthogonal transformation with a function with trivial mixed partial derivatives. 
\end{lemma}
\begin{proof}
 Because $C$ has orthogonal columns we can factor $C= B \Lambda$, $\Lambda\in \mathcal{D}$, $B\in \mathrm{SO}(n)$.
 From $\mathcal{L}^{2}_C\phi=C^{T}H\phi C=\Lambda B^{T}H\phi B \Lambda$
   and (\ref{eq:orthogonal-characteristics}) we deduce that 
 $B^{T}H\phi B\in \mathcal{D}(\Omega)$,
 or, equivalently, that $H\phi(Bx)\in \mathcal{D}(\Omega)$.
 Thus, locally 
 \[\phi(Bx)=\phi_1(x_1)+\cdots +\phi_n(x_n),\quad x=(x_1,\dots,x_n)\in B^{-1}(\Omega),\]
 and $\phi_i$ is strictly convex.  Therefore $\phi(Bx)$ is strictly convex and has property $\mathcal{I}$, and so the same occurs for
 $\phi(x)=\phi(B^{-1}(Bx))$.
\end{proof}

\section{The two-dimensional case}\label{sec:2dimensions}

We would like to know whether there exist strictly convex functions with property  $\mathcal{I}$ which do not have orthogonal characteristics. For that
we find 
convenient to discuss the algebraic structure of the system of PDEs (\ref{eq:inversible}). This should be easier in the lowest non-trivial dimension.

We shall denote partial derivatives of a function $\phi(x)$, $x=(x_1,x_2)\in \Omega\subset \R^{2}$, by means 
of subindices which follow a comma. We introduce independent variables to parametrize (homogeneous) jet spaces of order two and three:
%
%
\[\chi=\phi_{,11},\,\tau=\phi_{,12},\,\zeta=\phi_{,22},\,\upsilon=\phi_{,111},\,\nu=\phi_{,112},\,\omega=\phi_{,122},\,\xi=\phi_{,222}.\]
Strictly convex functions correspond to the open subset   $\chi\zeta-\tau^{2}>0$, $\chi+\zeta>0$. The system of PDEs (\ref{eq:inversible})
corresponds to the common solutions of the degree three homogeneous 
polynomials:
\begin{equation}\label{eq:jet-variety1}
 \begin{cases} 
   \xi(\chi\zeta-\tau^{2})-\zeta(\nu\zeta+\chi\xi-2\tau\omega)+\nu(\chi\zeta-\tau^{2})-\tau(\upsilon\zeta+\chi\omega-2\tau\nu) & =  0 \\
   -\omega(\chi\zeta-\tau^{2})+\tau(\nu\zeta+\chi\xi-2\tau\omega)-\upsilon(\chi\zeta-\tau^{2})+\chi(\upsilon\zeta+\chi\omega-2\tau\nu) & =  0
  \end{cases}
\end{equation}

\begin{lemma}\label{lem:jet-variety} Strictly convex functions with property $\mathcal{I}$ correspond in the space of jets of order three of functions
in the plane 
to an open subset of an intersection of quadrics:
\begin{equation}\label{eq:jet-variety2}
 \begin{cases} 
  (\zeta-\chi)\nu+\tau(\upsilon-\omega) & =  0 \\
 (\zeta-\chi)\omega+\tau(\nu-\xi) & =  0
  \end{cases},\quad \chi\zeta-\tau^{2}>0,\,\,\chi+\zeta>0
\end{equation}
\end{lemma}
\begin{proof}
We interpret the equations of the cubics (\ref{eq:jet-variety1}) as a (non-homogeneous) linear system with indeterminates $\nu\zeta+\chi\xi-2\tau\omega$
and $\upsilon\zeta+\chi\omega-2\tau\nu$:
\[ \begin{cases} 
  \zeta(\nu\zeta+\chi\xi-2\tau\omega)+\tau(\upsilon\zeta+\chi\omega-2\tau\nu) & =  (\xi+\nu)(\chi\zeta-\tau^{2}) \\
  \tau(\nu\zeta+\chi\xi-2\tau\omega)+\chi(\upsilon\zeta+\chi\omega-2\tau\nu) & =(\upsilon+\omega)(\chi\zeta-\tau^{2})   
  \end{cases}
  \]
In the open subset defined by $\chi\zeta-\tau^{2}\neq 0$ we obtain the equivalent relations
\[ \begin{cases} 
  \nu\zeta+\chi\xi-2\tau\omega & = \begin{vmatrix}
                                    \xi+\nu & \tau \\
                                   \upsilon+\omega & \chi
                                 \end{vmatrix} \\
   \upsilon\zeta+\chi\omega-2\tau\nu & =  \begin{vmatrix}
                                  \zeta &  \xi+\nu \\
                                  \tau & \upsilon+\omega
                                 \end{vmatrix} \\                              
  \end{cases}\Longleftrightarrow \begin{cases} 
  (\zeta-\chi)\nu+\tau(\upsilon-\omega) & =  0 \\
   (\zeta-\chi)\omega+\tau(\nu-\xi) & =  0
  \end{cases}
  \]
 \end{proof}

To each $[a:b]\in \mathbb{R}P^{1}$ one can associate the following second order PDE with constant coefficients for strictly convex functions:
\begin{equation}\label{eq:hyperbolic}
a \phi_{,11}-a \phi_{,22}=b \phi_{,12}.
\end{equation}
It corresponds to open subset of a hyperplane of the space of jets of order two
\begin{equation}\label{eq:hyp-pencil}
a(\chi-\zeta)-b \tau =0,\quad \chi\zeta-\tau^{2}> 0,\,\, \chi+\zeta>0
\end{equation}
whose  first prolongation is 
\begin{equation}\label{eq:first-prolongation}
\begin{cases} 
   a(\chi-\zeta)-b \tau & =0 \\
    a (\upsilon-\omega)-b \nu & =  0 \\
   a(\nu-\xi)-b \omega & =  0 
  \end{cases},\quad \chi\zeta-\tau^{2}> 0,\,\, \chi+\zeta>0.
\end{equation}


\begin{proposition}\label{pro:reg-sols} Let $\phi\in C^{\infty}(\Omega)$ be a strictly convex function.
\begin{enumerate}
 \item If $\phi$ satisfies  (\ref{eq:hyperbolic})
 for some $[a:b]\in \mathbb{R}P^{1}$, then $\phi$ has property $\mathcal{I}$.
 \item If $\phi$ satisfies  (\ref{eq:hyperbolic})
 for more than one $[a:b]\in \mathbb{R}P^{1}$, then $\phi$ is --- up to a degree one polynomial --- a multiple of the standard quadratic form
 $x_1^{2}+x_2^{2}$.
 \item If $\phi$ has property $\mathcal{I}$ and its Hessian has simple eigenvalues, then $\phi$ satisfies  (\ref{eq:hyperbolic})
 for some $[a:b]\in \mathbb{R}P^{1}$.
 \end{enumerate}
\end{proposition}
\begin{proof}
The set of equations (\ref{eq:hyperbolic}) are exactly those second order PDEs whose solutions have orthogonal characteristics. Therefore
item (1) is the specialization of Lemma \ref{lem:orthogonal-characteristics} to the two-dimensional case. Alternatively, 
item (1) follows from the inclusion of the solutions of  (\ref{eq:first-prolongation}) in the solutions of (\ref{eq:jet-variety2}).

If $\phi$ satisfies  (\ref{eq:hyperbolic})
 for more than one $[a:b]\in \mathbb{R}P^{1}$, then its second jet belongs to the base of the pencil (\ref{eq:hyp-pencil}): $\chi-\zeta=\tau=0$.
That is to say $\phi_{,12}=0$ and $\phi_{,11}=\phi_{,22}$. Therefore $0=\phi_{,221}=\phi_{,111}=\phi_{,112}=\phi_{,222}$.
Hence $\phi$ is a degree two polynomial whose homogeneous part of degree two equals $k(x_1^{2}+x_2^{2})$, $k>0$.

The Hessian $H\phi$ has two eigenvalues if and only if it misses the base of the pencil. 
Equivalently, the field of vectors in the plane $(\phi_{,22}-\phi_{11},\phi_{,12})\in \R^{2}$ has no zeroes. Therefore we can (locally) take the quotient 
of the components of the vector field to get a well-defined slope function. Property $\mathcal{I}$ 
as in (\ref{eq:jet-variety2}) can be rewritten
\[\begin{cases} \langle (\phi_{,22}-\phi_{11},\phi_{,12}),(\phi_{12},\phi_{,11}-\phi_{22})_{,1}\rangle & =0 \\
    \langle (\phi_{,22}-\phi_{11},\phi_{,12}),(\phi_{12},\phi_{,11}-\phi_{22})_{,2}\rangle & =0,
  \end{cases}\]
where $\langle\cdot,\cdot\rangle$ is the standard inner product. This 
implies that the slope function is constant, which
is exactly the second order PDE (\ref{eq:hyperbolic}) for some $[a:b]\in \mathbb{R}P^{1}$.
\end{proof}
Proposition \ref{pro:reg-sols} does not clarify whether strictly convex functions with property $\mathcal{I}$ and which do not have orthogonal characteristics
exist. As we shall discuss in  Section \ref{sec:bifurcations} such solutions exist: It is possible to start from  a multiple of the standard quadratic form in a subdomain of $\Omega$
  which  `bifurcates' into solutions to different equations in  (\ref{eq:hyperbolic})
  in other subsets of the domain $\Omega$.


\begin{remark}\label{rem:Cartan-connection} (The Cartan connection on jet spaces)
The algebraic manipulation in item (3) in Proposition \ref{pro:reg-sols} has a geometric counter-part.
The requirement  on the Hessian corresponds to the regularity condition needed to identify solutions with holonomic sections with respect to 
 the Cartan connection: 
On the one hand, the subset of the jet spaces which corresponds to property $\mathcal{I}$ is not smooth; the 1-forms
\[\begin{split}\Xi_1 &=\nu(d\zeta-d\chi)+(\zeta-\chi)d\nu+\tau(d\upsilon-d\omega)+(\upsilon-\omega)d\tau,\\
     \Xi_2 & =\omega(d\zeta-d\xi)+(\zeta-\xi)d\omega+
\tau(d\nu-d\epsilon)+(\nu-\epsilon)d\tau
\end{split}
\]
are colinear in the subset $\omega(\upsilon-\omega)-\nu(\nu-\epsilon)=\chi-\zeta=\tau=0$.
On the other hand, the smooth locus of the intersection of quadrics fails to be transverse to the 
fibers of the projection onto jets of order two in the points over the base of the pencil.

The connection 1-forms one has to add when passing from jets of order two to jets of order three are:
\[\Theta_1=d\xi-\upsilon dx-\nu dy,\quad \Theta_2=d\tau-\nu dx-\omega dy,\quad \Theta_3=d\chi-\omega dx-\zeta dy.\]
The pullback foliation is defined by the 1-forms
\[K_1=(\xi-\chi)d\theta-\tau(d\xi-d\chi),\quad K_2=(\upsilon-\omega)d\nu-\nu(d\upsilon-d\omega),\quad K_3=(\nu-\epsilon)d\omega-\omega(d\nu-d\epsilon).\]
The equalities 
\[K_1=(\chi-\zeta)\Theta_2-\tau(\Theta_1-\Theta_3),\quad K_2-\frac{\nu^{2}}{\tau^{2}}K_1=-\frac{\nu}{\tau} \Xi_1,\quad 
K_3-\frac{\omega^{2}}{\tau^{2}}K_1=-\frac{\omega}{\tau} \Xi_2
\]
hold in the intersection of (\ref{eq:jet-variety2}) with the complement of the pullback of the base of the pencil.
Therefore holonomic sections in this subset are tangent to the pullback foliation. Hence their order two jet must be inside a hyperplane of the pencil.
\end{remark}

\begin{remark}\label{rem:legendre} (Orthogonal characteristics and Legendre transform) Let $\phi$ be an strictly 
convex function on a convex domain $\Omega\subset \R^{2}$. Its Legendre transform is an strictly convex function $\phi^{*}$
on a convex domain $\Omega^{*}$ which is related to $\Omega$ by a (Legendre) diffeomorphism. Let us assume that $\phi$ satisfies the
constant coefficients second order PDE: 
\[a \phi_{,11}+c \phi_{,22}-b \phi_{,12}=0,\quad [a:c:b]\in \mathbb{R}P^{2}.\]
Because $H\phi^{*}$ at $x\in\Omega^{*}$ equals $H\phi^{-1}$ at its related point in $\Omega$ we have the equality
\[a \phi_{,22}^{*}+c \phi_{,11}^{*}+b \phi_{,12}^{*}=0.\]
Thus the Legendre transform induces an involution in the parameter space of constant coefficients degree two homogeneous PDEs:
$[a:c:b]\mapsto [c:a:-b]$.
Its fixed point set is 
$[1:1:0]\cup [a:-a:b]\subset \mathbb{R}P^{2}$.
To the point $[1:1:0]$ corresponds the Laplace equation, which has no strictly convex solutions. The projective line
$[a:-a:b]$ parametrizes hyperbolic PDEs with orthogonal characteristics (\ref{eq:hyperbolic}).
Therefore if $\phi$ is a function  on the convex domain $\Omega$ with orthogonal characteristics, so its Legendre transform is.
This invariance property holds regardless of the 
dimension:
\[\mathcal{L}^{2}_C\phi\in \mathcal{D}(\Omega)\Longleftrightarrow  C^{T} (H\phi^{*})^{-1}C \in D(\Omega^{*})\Longleftrightarrow  
 C^{T} (H\phi^{*})C \in \cD(\Omega^{*}),
\]
where the first equivalence uses the relation between Hessian matrices of the original function and its Legendre transform,
and in the second equivalence 
we have inverted the matrices and we have used $C^{T}C\in \mathcal{D}$. 
\end{remark}

\section{Hessian metrics with symmetric Christoffel symbols}\label{sec:Hessian}
To generalize the results in Section \ref{sec:2dimensions} 
the complexities brought by the increase of dimension shall be dealt with by shifting the perspective to that of Hessian metrics.

A {\bf Hessian metric} on $\Omega\subset \R^{n}$ is a Riemannian metric obtained as the Hessian matrix of a (strictly convex) function 
on $\Omega$. Property $\mathcal{I}$ for strictly convex functions can be translated to a requirement for a Hessian metric: that its inverse
metric be Hessian as well. In such case we say that the given Hessian metric has property $\mathcal{I}$.

There is another natural differential condition on Hessian metrics which allows to formulate  in arbitrary dimensions the algebraic simplification
of property $\mathcal{I}$ described in Lemma \ref{lem:jet-variety}.  
For a Hessian metric $H\phi$ the Christoffel symbols of the first kind equal the partial derivatives of order three:
$\Gamma_{ijk}=\phi_{,ijk}$.
The Christoffel symbols (of the second kind) are
\[\Gamma_{kj}^{i}={H\phi^{-1}}_{il}\Gamma_{ljk}.\]
Let $[H\phi]_{,k}$ denote the partial derivative with respect to $k$ of the entries of the Hessian matrix. For each $1\leq k\leq n$ 
we define the \emph{Christoffel matrix} 
\[\Gamma_k:=\Gamma_{k \star }^{\bullet}={H\phi^{-1}}_{\bullet\circ}{[H\phi]_{,k}}_{\circ\star}=H\phi^{-1}[H\phi]_{,k}.\]

\begin{definition}\label{def:fully-symmetric} A Hessian metric $H\phi$ on $\Omega$
has {\bf symmetric Christoffel symbols} if the 
Christoffel symbols (of the second type) are symmetric on the three indices. Equivalently, if its Christoffel matrices are symmetric.
\end{definition}
%

Property $\mathcal{I}$ corresponds to an open subset of the solutions of a system of polynomial equations of degree $2n-1$ in the space
of jets of order three. The symmetry of the Christoffel symbols is determined by a system of polynomial equations of degree $n$; for $n=2$
it is exactly (\ref{eq:jet-variety2}).
The generalization of Lemma \ref{lem:jet-variety} to arbitrary dimensions is that property $\cI$ translates
into the symmetry of Christoffel symbols:
\begin{proof}[Proof of Lemma \ref{lem:propertyI-equals-symmetric}]
The Hessian metric is invertible if and only if the $i$-th and $j$-th lines of 
$[H\phi^{-1}]_{,j}$ and $[H\phi^{-1}]_{,i}$ are equal. This is equivalent to the same condition for the matrices
$[H\phi^{-1}]_{,j}H\phi$ and $[H\phi^{-1}]_{,i}H\phi$. 
If we prolong the identity 
$H\phi^{-1}H\phi=I$,
then the condition transforms on the same condition for the Christoffel matrices $\Gamma_j$ and $\Gamma_i$.
This amounts to symmetry of all Christoffel matrices.
\end{proof}

The problem of the symmetry of Christoffel symbols of Hessian metrics is amenable to Lie theoretic methods. A first instance of that is the following:

\begin{proposition}\label{pro:sym-refined} The following statements for a Hessian metric $H\phi $ on $\Omega$  are equivalent:
\begin{enumerate}
 \item It has symmetric Christoffel symbols.
 \item The two matrices $H\phi$ and $[H\phi]_{,k}$ are in the same Cartan subalgebra of the symmetric matrices,  $1\leq k\leq n$. (The Cartan subalgebra may 
 vary with $k$).
 \item The two matrices $H\phi$ and $\Gamma_k$ are in the same Cartan subalgebra of the symmetric matrices, $1\leq k\leq n$. (The Cartan subalgebra may 
 vary with $k$).
\end{enumerate}
 \end{proposition}
\begin{proof}
Let $\ss$ be the vector subspace of symmetric matrices and let $\mathfrak{d}\subset \ss$ denote the diagonal matrices; this is a maximal Cartan
subalgebra of the symmetric matrices. 

The Christoffel matrix $\Gamma_k$ is the product of the symmetric matrices 
 $H\phi^{-1}$ and $[H\phi]_{,k}$. Therefore $H\phi$ has symmetric Christoffel matrices if and only if the following commutators are trivial:
 \[[H\phi^{-1},[H\phi]_{,k}]=0,\quad 1\leq k\leq n.\]
 This is equivalent to require that $[H\phi]_{,k}$ be in the same maximal torus of $\ss$ as $H\phi^{-1}$. 
 If $B$ is a orthogonal matrix which diagonalizes $H\phi$ then it also diagonallizes $H\phi^{-1}$:
 \[B^{T}H\phi B=\Lambda,\quad B^{T}H\phi ^{-1}B=\Lambda^{-1}.\]
   Therefore if $H\phi, [H\phi]_{,k}$ are in the Cartan subalgebra $\mathrm{Ad}_B(\mathfrak{d})\subset \ss$, then 
   so is $H\phi^{-1}$. This shows
 the equivalence between $(1)$ and $(2)$.
 
 If $(2)$ holds then 
 \[\Gamma_k=H\phi^{-1}[H\phi]_{,k}=B^{T}\Lambda_1BB^{T}\Lambda_2B=B^{T}\Lambda_1\Lambda_2B\]
 remains in the same Cartan subalgebra of the commuting factors, which 
 proves $(3)$. Condition $(3)$ by definition implies the symmetry of the Christoffel matrices. 
\end{proof}

By Lemma \ref{lem:propertyI-equals-symmetric}  Hessian metrics with symmetric Christoffel symbols are the same as Hessian metrics with 
property $\mathcal{I}$. Thus by Lemma \ref{lem:orthogonal-characteristics} strictly convex functions with orthogonal characteristics define  Hessian metrics with symmetric
Christoffel symbols. We can reprove
this result with a Lie theoretic approach: 

\begin{lemma}\label{lem:orthogonal-to-symmetric} 
 Let $\phi\in C^{\infty}(\Omega)$.
If $\mathcal{L}^{2}_C\phi\in \cD(\Omega)$ for some $C\in \mathcal{C}$, then the Christoffel matrices of $H\phi$ 
for all points in $\Omega$ are in the Cartan subalgebra 
$\mathrm{Ad}_C(\mathfrak{d})$.
In particular $H\phi$ has symmetric Christoffel symbols. 

\end{lemma}
\begin{proof}
By hypotheses for each $x\in \Omega$ 
\[C^{T}H\phi C=\Lambda,\quad \Lambda=\Lambda(x)\in \mathcal{D}(\Omega), \quad C\in \mathcal{C}.\]
Hence $H\phi={(C^{T})}^{-1}\Lambda C^{^{-1}}$ and upon taking its first order prolongation
\[[H\phi]_{,k}={(C^{T})}^{-1}\Lambda_{,k} C^{-1}.\]
Therefore both $H\phi$ and $[H\phi]_{,k}$ are in  $\mathrm{Ad}_C(\mathfrak{d})$.
By item (2) in Proposition \ref{pro:sym-refined} the same occurs for $\Gamma_k$ 
(the action by conjugation on $\mathfrak{d}$ of second factor of $\mathcal{C}=\mathrm{SO}(n)\mathcal{D}$  is trivial). 
By Proposition \ref{pro:sym-refined} this implies symmetry of Christoffel matrices. 
\end{proof}


\begin{proposition}\label{pro:legendre}  The Legendre transform preserves the class of Hessian metrics with property $\mathcal{I}$ on convex domains.
 \end{proposition}
\begin{proof}
Let $\mathcal{L}_j$ denote the Lie derivative with respect to $\tfrac{\partial}{\partial x_j}$.
We can rewrite property $\mathcal{I}$ for $H\phi$ as 
\begin{equation}\label{eq:inversibility-Lie}
\mathcal{L}_{\bullet}{(H\phi^{-1})}_{\star k}-\mathcal{L}_{\star}{(H\phi^{-1})}_{\bullet k}=0.
\end{equation}
The differential of $\phi$ defines the Legendre diffeomorphism 
\[d\phi:\Omega\to \Omega^{*}, \quad D(d\phi)=H\phi.\]
If we push forward each equation in (\ref{eq:inversibility-Lie} by the Legendre diffeomorphism $d\phi$, then the Lie derivative of the 
pushed forward functions --- entries of the inverse Hessian --- by the pushed forward vector fields will subtract to zero as well.
The entries of the inverse Hessian matrix  are pushed forward to the entries of the Hessian of $\phi^{*}$;  
the coordinate vector fields are pushed forward to the columns vector fields of the Jacobian matrix $H\phi$, which at points
in $\Omega^{*}$ is the matrix ${H\phi^{*}}^{-1}$.
Therefore property $\mathcal{I}$ for $H\phi$ is equivalent to
\[{{H\phi^{*}}^{-1}}_{\bullet\circ}{[H\phi^{*}]_k}_{\circ\star}-
{{H\phi^{*}}^{-1}}_{\star\circ }{[H\phi^{*}]_k}_{\circ\bullet}=0,\]
which is the symmetry of the Christoffel matrices of $H\phi^{*}$. Therefore 
 by Lemma \ref{lem:propertyI-equals-symmetric} $H\phi^{*}$ has property $\mathcal{I}$.
\end{proof}

\section{The universal frame bundle for Hessian metrics and matrices with orthogonal columns}\label{sec:frame-bundle}

To generalize Proposition \ref{pro:reg-sols} to arbitrary dimensions jet spaces will be replaced by (a subset of) the principal orthogonal frame bundle of the Hessian 
metric with its Levi-Civita connection. There are three reasons to do that: 
\begin{enumerate}[(a)]
\item A function $\phi\in C^{\infty}(\Omega)$ has property $\mathcal{I}$ if and only if $H\phi$ has symmetric Christoffel symbols. The Christoffel 
symbols are the components of the Levi-Civita connection. Therefore one may expect a reformulation of property $\mathcal{I}$ related to the 
tangent or orthogonal frame bundle with the Levi-Civita connection. 
 \item  If a function $\phi\in C^{\infty}(\Omega)$ has orthogonal characteristics, then the Hessian metric $H\phi$ splits (locally, but along the 
 same characteristics everywhere). In other words, the conclusion of the de Rham 
 Splitting Theorem holds. Therefore to study the relation between property
 $\mathcal{I}$ and orthogonal characteristics it may be appropriate to look at parallel transport  
 on  the principal frame bundle with its the Levi-Civita connection.
 \item  Hessian metric on domains of Euclidean space are characterized among Riemannian metrics as those 
 whose frame bundle is the pullback of a universal principal bundle with connection coming from 
 symmetric space theory \cite{Du}. 
\end{enumerate}

For a function $\phi$ the information of the homogeneous part of its second jet is the same as the one contained in its Hessian.
Thus for our strictly convex functions we shall be looking at the map $x\mapsto H\phi(x)$, which takes values in the positive matrices $\cP$.
There, the pencil in (\ref{eq:hyp-pencil}) defined by hyperbolic PDEs with orthogonal characteristics generalizes as follows:
The second order PDE equation (\ref{eq:axis-characteristics}) corresponds to Hessian metrics with image
in the (positive) diagonal matrices $\mathcal{D}$. Matrices with orthogonal columns have a
factorization into an orthogonal and a diagonal matrix. Thus we may confine ourselves to the family of second order PDEs
$\cL^{2}_B\phi\in \mathcal{D}$, $B\in \mathrm{SO}(n)$.
To each of them there corresponds the subset $\mathrm{Ad}_{B}(\mathcal{D})\subset \cP$; their union over $B\in \mathrm{SO}(n)$ fills $\cP$
as any positive matrix can be diagonalised by a special orthogonal transformation. 

For a Riemannian metric defined on a subset of Euclidean space, its orthogonal frame bundle --- forgetting for the moment about the Levi-Civita connection ---
is constructed via pullback:
The map $\pi: \mathrm{Gl}(n)^{+}\to \mathrm{Gl}(n)^{+}$, $A\mapsto {A^{-1}}^{T}A^{-1}$,
has image the closed embedded submanifold of positive matrices. The restriction to its image 
\begin{equation}\label{eq:Hessian-Cartan}
\pi:\mathrm{Gl}(n)^{+}\to \mathcal{P}
\end{equation}
  \begin{itemize}
   \item is a (right) principal bundle for $\mathrm{SO}(n)$;
   \item intertwines the right action of $\mathrm{SO}(n)$ on $\mathrm{Gl}(n)^{+}$ and the adjoint action of $\mathrm{SO}(n)$ on $\mathcal{P}$;
   \item is the bundle of (positively oriented) orthogonal frames for metrics on $\R^n$.
  \end{itemize}

%
%
%
%
%

Let  $\nabla$ be the  $\mathrm{SO}(n)$-invariant principal connection on $\pi:\Gl\to \cP$ which  at the identity matrix has as horizontal 
space the symmetric
matrices\footnote{Its curvature there is $[\ss,\ss]$.}.
\begin{proposition}\cite[Proposition 4.1]{Du}\label{pro:universal}
If $H\phi$ is a Hessian metric on $\Omega$, then the pullback of  $\nabla$  by 
$H\phi:\Omega\to \cP$, $x\mapsto H\phi(x)$,
is the Levi-Civita connection on the orthogonal frame bundle of $H\phi$. Furthermore, this property characterizes Hessian metrics
among Riemmanian metrics in domains of Euclidean space.


\end{proposition}

The appropriate replacement of the jets of order three will not be the full bundle of orthogonal frames. It will  be 
the subset of matrices with orthogonal columns. The following result, of which  Proposition \ref{pro:Cartan} in the Introduction follows,
shows that it is  well-behaved with respect to the universal Levi-Civita
connection:

%
%
%
\begin{proposition}\label{pro:non-linear-inv} The subset of  matrices with orthogonal columns
$\cC\subset \Gl$ has the following properties:
%
\begin{enumerate}
 \item It is a closed embedded submanifold of $\Gl$ on which the
 Cartan decomposition $\Gl=\mathrm{SO}(n)\cP$ induces a product structure $\cC=\mathrm{SO}(n)\mathcal{D}$.
\item  The intersection of the horizontal distribution of $\nabla$ with the tangent bundle $T\cC$  
 is an involutive distribution on $\cC$. Its foliation $\cF$
 is the one associated to the Cartan decomposition, with leaves
  the left $\mathrm{SO}(n)$-translates of $\mathcal{D}$.

\end{enumerate}

\end{proposition}
\begin{proof}
Let $\iota$ be the inversion map on $\Gl$ and $q=\pi\circ \iota:\Gl\to \cP$, $A\mapsto A^{T}A$.
A matrix $C$ has orthogonal columns if and only if $q(C)\in \mathcal{D}$. 
Therefore $\cC$ is the preimage under a submersion of the closed embedded 
submanifold of positive diagonal matrices, thus  
 a closed embedded submanifold of $\Gl$.
We have already used the (unique) factorisation of a matrix with orthogonal columns as a product of an orthogonal and a diagonal matrix. It is 
straightforward that it gives rise to a Cartesian product of manifolds  $\cC=\mathrm{SO}(n)\mathcal{D}$.

The product structure in (1) implies that its tangent space at $C\in \mathcal{C}$ is 
\[\mathfrak{so}(n)\cdot C \cdot \mathfrak{d}=C\cdot  \mathrm{Ad}_{C^{-1}}(\mathfrak{so}(n))\cdot \mathfrak{d}.\]
The horizontal space of $\nabla$ there is $C\cdot \ss$. Because the conjugation of a skew orthogonal matrix by an orthogonal one can never be symmetric, 
the intersection of the tangent spaces must be $C\cdot \mathfrak{d}$. Therefore the intersection of the horizontal space
of $\nabla$ with $T\cC$ is the distribution  \footnote{One could also deduce involutivity
by recalling that the curvature of the connection is $C\cdot [\ss,\ss]$, 
and, therefore,  the abelian subalgebra $\mathfrak{d}$ is flat.} 
tangent to the left translates of $\cD$ by $\mathrm{SO}(n)$.
%
%
%
\end{proof}

Next, we argue  how $\pi:(\cC,\cF)\to \cP$ provides a `desingularization' of the pencil $\Ad{_B}(\cD)$, $B\in \mathrm{SO}(n)$.
\begin{proposition}\label{pro:desingularisation} The restriction $\pi|_{\cC}:\cC\to \cP$ has the following properties:
\begin{enumerate}
 \item It is a  surjective map all whose values are clean.
  \item The restriction of the differential of $\pi|_{\cC}$ to $T\cF$ has trivial kernel and the restriction of $\pi|_{\cC}$
  to the leaf $B\cD$ is a diffeomorphism onto $\Ad{_B}(\cD)$.
\end{enumerate}
\end{proposition}
\begin{proof}
Let $V\in \cP$. Then it diagonalizes in an orthogonal basis: 
$B^{T}V B=\Lambda$, $B\in \mathrm{SO}(n)$, $\Lambda\in \cD$.
Hence $\pi(B\Lambda^{1/2})=V$, so  $\pi|_{\cC}$ is surjective. 
The fiber is 
\[\pi|_{\cC}^{-1}(V)=B\Lambda^{1/2}\mathrm{SO}(n)_\Lambda,\]
where the latter subgroup is the stabilizer of $\Lambda$ for the adjoint action.
The kernel of the differential of $\pi$ at $B\Lambda^{1/2}$ is
$B\Lambda^{1/2}\cdot \mathfrak{so}(n)$. The tangent space of $\cC$ at $B\Lambda^{1/2}$
is $B\Lambda^{1/2}\cdot \mathrm{ad}_\Lambda^{-1}(\mathfrak{d})$. Because the adjoint orbit 
through $\Lambda$ intersects $\cD$ cleanly at $\Lambda$, their intersection --- which is the kernel of 
 of the differential of $\pi|_{\cC}$ at $B\Lambda^{1/2}$ --- is $B\Lambda^{1/2}\cdot \mathfrak{so}(n)_\Lambda$.
Therefore all values of $\pi|_{\cC}$ are clean.

The tangent space to the leaf of $\cF$ through $B\Lambda^{1/2}$ is $B\Lambda^{1/2}\cdot \mathfrak{d}$. Its intersection with 
$B\Lambda^{1/2}\cdot \mathfrak{so}(n)_\Lambda$ is trivial. Therefore the restriction of $\pi$ to $B\cD$ is a local diffeomorphism 
over its image. That image is by construction $\mathrm{Ad}_B(\cD)$. To conclude that it is a diffeomorphism 
one can either check that the map is bijective or argue that the manifolds 
involved are contractible.
\end{proof}

The base of the pencil (\ref{eq:hyp-pencil}) corresponds to inner products in the plane which have a unique eigenvalue. In arbitrary dimensions
we have an analogous subsets. For each symmetric matrix we can order its eigenvalues (with their multiplicity)  in an increasing sequence. To each partition 
$\kappa$ of $\{1,\dots,n\}$ there correspond a subset $\Theta^{\kappa}_\ss$; likewise, to each matrix with orthogonal columns we can order
the norm of its columns in an increasing sequence. In that way we obtain partitions of $\cD$, $\cP$ and $\cC$:
$\Theta_{\cD}$, $\Theta_\cP$, $\Theta_{\cC}$.

%
%

\begin{proposition}\label{pro:desingularisation2}
 The partitions $\Theta_{\cD},\Theta_\cP$ and $\Theta_{\cC}$ are stratifications of $\cD$, $\cP$ and $\cC$, respectively,
and they interact with the map $\pi|_{\cC}:(\cC,\cF)\to \cP$ as follows:
\begin{enumerate}
 \item The preimage of $\Theta_\cP^{\kappa}$ is $\Theta_{\cC}^{\kappa}$ and the restriction  is a principal bundle:
 \begin{equation}\label{eq:stratum-submersion} 
\pi|_{\Theta^{\kappa}_{\cC}}: \Theta^{\kappa}_{\cC}\to \Theta^{\kappa}_\cP.
 \end{equation}
 \item The foliation $\cF=\mathrm{SO}(n)\cD$ of $\cC$ intersects the stratum $\Theta^{\kappa}_{\cC}$ cleanly and induces there the foliation  
 $\mathrm{SO}(n)\Theta^{\kappa}_\cD$ of $\Theta^{\kappa}_{\cC}$.
 \item The foliation $\mathrm{SO}(n)\Theta^{\kappa}_\cD$ of $\Theta^{\kappa}_{\cC}$ is projectable by the submersion $\pi|_{\Theta^{\kappa}_{\cC}}$. Its image  
 is the foliation $\mathrm{Ad}_{\mathrm{SO}(n)}(\Theta_\cD^{\kappa})$ of $\Theta^{\kappa}_\cP$.
\end{enumerate}
 It is in this sense that $\pi:(\cC,\cF,\Theta_{\cC})\to (\cP,\mathrm{Ad}_{\mathrm{SO}(n)}(\cD),\Theta_{\cP})$ is a desingularization of the
 stratified pencil.
\end{proposition}
\begin{proof}
The group $\mathrm{SO}(n)$ acts on $\ss$ by conjugation. As for any proper action it produces a stratification of $\ss$ in orbit types \cite[Chapter 2]{DK}: two symmetric matrices are related if 
their isotropy subgroups are conjugated. It is well known that upon passing to connected components the outcome is a (Whitney B) stratification 
of $\ss$. The stratification $\Theta_\ss$ is the result of possibly collecting some of the strata of the orbit type stratification belonging to 
the same subset of the orbit type partition; in any case, it is 
still a stratification for the partial order associated to the partitions $\kappa$ of $\{1,\dots,n\}$. The stratification 
$\Theta_\ss$ --- made of adjoint orbits --- intersects the Cartan subalgebra $\mathfrak{d}$ cleanly, thus inducing a stratification $\Theta_\mathfrak{d}$.
Each strata there is a face of the positive Weyl chamber of diagonal matrices with diagonal elements ordered increasingly; an open 
convex polytope in a vector subspace $\mathfrak{d}^{\kappa}$.
The partition $\Theta_\cP$ is obtained by intersecting $\Theta_\ss$ with the open subset of positive matrices, thus it is a stratification. 
The partition $\Theta_\cD$ is also obtained upon intersection; it is a stratification because for instance $\cD$ is an open subset of $\mathfrak{d}$.
Finally, $\Theta_{\cC}$ is the  pull back of $\Theta_{\cD}$ by the submersion $q$, and therefore it is a stratification as well.

Let $C\in \cC$ with factorisation $C=B\Lambda$. Then $q(C)=B\Lambda^{-2}B^{T}$ and therefore
\[C\in \Theta^{\kappa}_{\cC}\Longleftrightarrow \Lambda \in \Theta^{\kappa}_{\cD}
\Longleftrightarrow \Lambda^{-2} \in \Theta^{\kappa}_{\cD}\Longleftrightarrow q(C) \in \Theta^{\kappa}_{\cD}.
\]
The stratum $\Theta_k^{\cD}$ is an open subset of the vector subspace $\mathfrak{d}^{\kappa}$ of all matrices whose stabilizer contains
  $\mathrm{SO}(n)_{\kappa}$. Because the fiber of $\pi|_{\cC}$ through $C$ is $C\mathrm{SO}(n)_{\kappa}$ and 
$\pi|_{\Theta^{\kappa}_{\cC}}$ is saturated by fibers of $\pi|_{\cC}$, 
 it is a principal $\mathrm{SO}(n)_{\kappa}$-bundle. This proves (1).


The fibers of $q$ are the orbits of the left $\mathrm{SO}(n)$-action. The restriction $q|_{\cD}:\cD\to \cD$ is the square map, which 
preserves the strata of $\Theta_{\cD}$. Therefore the factorization of $\cC$ is compatible with the stratification:
\[\Theta_{\cC}=\mathrm{SO}(n)\Theta_{\cD}.\]
Thus the intersection of the leaf of $\cF$ though $C\in \Theta_{\cC}^{\kappa}$ is $C\Theta_{\cD}^{\kappa}$.
At $C$ the respective tangent spaces are $C\cdot \mathfrak{d}$ and $C\cdot \mathfrak{d}^{\kappa}$.
Therefore the intersection is clean and  this proves (2).

By item (1) above and by item (2) in Proposition \ref{pro:desingularisation} the restriction of
$\pi|_{\Theta^{\kappa}_{\cC}}$ to the leaf $C\cdot \Theta^{\kappa}_\cD\subset \Theta^{\kappa}_{\cC}$ is a diffeomorphism 
over its image. Its image is $\mathrm{Ad}_{B}(\Theta^{\kappa}_{\cD})$ ($\pi|_{\cD}:\cD\to \cD$ is the inverse of the 
square map); it is in fact the common image of all leaves through points of 
the fiber $C\mathrm{SO}(n)_\kappa$. 
\end{proof}

We can now sharpen Proposition \ref{pro:sym-refined}.
\begin{proposition}\label{sym-refined-stratum}
Let $H\phi$ be a Hessian metric on $\Omega$  such that $H\phi(\Omega)$ is contained in the stratum $\Theta^{\kappa}_\cP$.
Then its Christoffel symbols are symmetric if and only if 
 $H\phi$ and $[H\phi]_k$ can be conjugated by an orthogonal matrix to a 
matrix in $\mathfrak{d}^{\kappa}$,  $1\leq k\leq n$.
\end{proposition}
\begin{proof}
Because the image of $H\phi$ is contained in  $\Theta^{\kappa}_\cP$, its partial derivatives must be in the tangent space
to the stratum: $[H\phi]_k\in T\Theta^{\kappa}_\cP$.
By item (2) in Proposition \ref{pro:sym-refined} there exists a special orthogonal matrix $B$ which conjugates $H\phi$ and $[H\phi]_k$ to a diagonal one:
$B^{T}[H\phi]_k B\in \mathfrak{d}$.
Therefore
\[B^{T}[H\phi]_kB\in \mathfrak{d}\cap T\Theta^{\kappa}_\cP=\mathfrak{d}^{\kappa}.\]
\end{proof}

%

\section{Differential relations on the submanifold of matrices with orthogonal columns}\label{sec:diferential-relation}

We want to transfer property $\cI$ for Hessian metrics into a differential condition for the orthogonal frame bundle at 
the submanifold of matrices with orthogonal columns. 

Let $H\phi$ be a Hessian metric on $\Omega\subset \R^{n}$. To every curve 
$\gamma:(-\epsilon,\epsilon)$ based at $x\in \Omega$ we associate a curve in $\cP$ based at $H\phi(x)$:
\[H\phi(\gamma):(-\epsilon,\epsilon)\to \cP,\quad t\mapsto H\phi(\gamma(t)).\]
Upon choosing an orthonormal frame for $H\phi(x)$,  we can construct the horizontal lift of $H\phi(\gamma)$ based at the orthonormal frame.
\begin{definition}\label{def:tangent-horizontal-lift} A Hessian metric $H\phi$ in $\Omega$ has
property $\cC$ if for any point $x\in \Omega$ and any curve 
$\gamma$ at $x\in \Omega$ there exist an orthonormal frame $C\in \cC$ for $H\phi(x)$ such that the corresponding
horizontal curve is tangent to $\cC$ at $C$.
\end{definition}

We now translate property $\cI$ to the universal orthogonal frame bundle setting:

\begin{proof}[Proof of Theorem \ref{thm:differential-relation}]
We must show that  a Hessian metric in $\Omega\subset \R^{n}$  has property $\cC$ if and only if it has symmetric Christoffel symbols.

Property $\cC$ is linear in the velocity of the curve at $x$.
Thus it is enough to prove the equivalence for $\gamma(t)=x+te_k$, $1\leq k\leq n$.
Let us denote the horizontal lift at $A\in \cC$ by $A(t)$.
That $A$ belongs to $\cC$ means that  $A^{T}A=\Lambda\in \cD$.
By Proposition \ref{pro:universal} (taken from  \cite{Du}) the pullback of $\pi:(\Gl,\nabla)\to \cP$ by $H\phi$ is the  orthonormal frame bundle 
of $H\phi$ with its Levi-Civita
connection. Thus we have:
\[0={A'}_{\bullet \star}+\Gamma_{k \circ }^{\bullet}A_{\circ \star}(=A'+\Gamma_k A).\]
The image by the differential of $q$ of the vector of $A'=A'(0)$ is ${A^{T}}'A+A^{T}A'$.
Therefore the Hessian metric satisfies property $\cC$ at $A$ if and only if
\begin{equation}\label{eq:differential-relation}
A^{T}\Gamma_{k}^{T}A+ A^{T}\Gamma_{k}A\in \mathfrak{d}.
\end{equation}
We have $\Gamma_k={H\phi^{-1}}{[H\phi]_{,k}}$, $H\phi^{{-1}}=AA^{T}$,
where the latter identity  uses that $A$ is an orthonormal frame for $H\phi$.
Hence we may rewrite $\Gamma_{k}=AA^{T}[H\phi]_{,k}$.
Thus equation (\ref{eq:differential-relation}) is equivalent to:
\[A^{T}[H\phi]_{,k}AA^{T}A+A^{T}AA^{T}[H\phi]_{,k}A=A^{T}[H\phi]_kA\Lambda +\Lambda A^{T}[H\phi]_{,k}A\in \mathfrak{d}.\]
Because $\Lambda$ has non-zero positive entries 
if its anti-commutator with a  matrix  is diagonal, then the  matrix must be diagonal.
The conclusion is that  property $\cC$ is equivalent to:
\[A^{T}[H\phi]_{,k}A\in \mathfrak{d},\quad A^{T}A\in \cD,\quad AA^{T}=H\phi^{-1}.\]
By item (2) in Proposition \ref{pro:sym-refined} this is exactly the symmetry of the Christoffel matrices.
%
%
%
\end{proof}

%
 
We can verify that strictly convex functions with orthogonal characteristics satisfy property $\cC$. 
\begin{lemma} 
If $\cL^{2}_C\phi \in \cD$, $C\in \cC$,
then $H\phi$ satisfies property $\cC$.
\end{lemma}
\begin{proof}
By definition $C^{T}H\phi C=\Lambda$, $\Lambda\in \cD$.
Therefore $C^{T}[H\phi]_{,k}C\in \mathfrak{d}$.
The matrix $C\Lambda^{-1/2}$ also belongs to $\cC$ and it is an orthonormal frame for $H\phi$.
Therefore
\[{(C\Lambda^{-1/2}})^{T}[H\phi]_{,k}C\Lambda^{-1/2}\in \mathfrak{d},\]
and thus property $\cC$ holds.
\end{proof}

Next we analyze up to which extent Hessian metrics with property $\cC$ are defined by functions with orthogonal characteristics.

\begin{definition}\label{def:Cartan_Kahler} 
 A Hessian metric $H\phi$ on $\Omega\subset \R^{n}$
 has property  $\mathcal{CK}$ if there exists a point $x\in \Omega$ and an orthonormal frame 
 $C\in \cC$ for $H\phi(x)$, such that for every curve in $\Omega$ based at $x$ its horizontal lift
 at $C$ is contained in $\cC$.
\end{definition}


\begin{theorem}\label{thm:Cartan-Kahler}  If a  Hessian metric $H\phi$ on $\Omega$ 
 has property  $\mathcal{CK}$, then it solves  $\cL^{2}_C\phi \in \mathfrak{d}$, $C\in \cC$.
 
 Furthermore, the following conditions are equivalent:
 \begin{enumerate}
 \item The image $H\phi(\Omega)$ is contained in the stratum $\Theta_{\cP}^{\kappa}$.
 \item Property $\mathcal{CK}$ holds for all orthonormal frames in $\cC$ over all points 
 of $H\phi(\Omega)$.
  \end{enumerate}
 In either case $\phi$ restricts to the leaves of the (parallel) foliation determined by the (rotation of) the subspace $\mathfrak{d}^{\kappa}$
 to a multiple 
 of the standard quadratic form (up to an affine summand).
 

\end{theorem}
\begin{proof}
Let $x$ and $C$ be a point and orthonormal frame for $H\phi(x)$ with respect which condition $\mathcal{CK}$ is defined.
Because $\Omega$ is (path) connected for every point $y$ we have a curve $\gamma$ starting at $x$ and ending at $y$.
By hypotheses the horizontal lift of $H\phi(\gamma)$ is a curve $A(t)$ contained in $\cC$.
Therefore it is in $T\cC$ and horizontal. By (2) in Proposition \ref{pro:non-linear-inv} $A(t)$
is contained in the leaf $C\cD$ of $\cF$:
$A(t)=CD(t)$.
Because $A(t)$ is a curve of orthogonal frames  $A(t)^{T}H\phi (\gamma)A(t)=\mathrm{I}$.
Therefore
\[C^{T}H\phi C=D(t)^{-2}\in \cD.\]
Equivalently, if $C=B\Lambda$ then by  (2) in Proposition \ref{pro:desingularisation}
$\pi|_\cC$ sends the leaf $C\mathcal{D}$ 
diffeomorphically onto $\mathrm{Ad}_{B}(\mathcal{D})$, which is where the image of $H\phi$ must be confined.
The (local) splitting condition for $H\phi$ along the characteristics of $C$ can be also argued as follows:
that $A(t)\subset C\mathcal{D}$ implies that line fields at $x$ spanned by each column of $C$ are invariant by parallel transport. Thus
the de Rham Splitting Theorem applies (and Hessian metrics restrict to Hessian metrics).

If $H\phi\subset \Theta^{\kappa}_\cP$,  then by item (1) in Proposition \ref{pro:desingularisation2} the
frame $C\in \cC$ with respect to which $\mathcal{CK}$ is defined belongs to $\Theta^{\kappa}_{\cC}$. By item (2) in the same Proposition all horizontal curves
based at $C$ must be contained in the leaf $C\Theta^{\kappa}_{\cD}$ of the foliation of $\Theta^{\kappa}_{\cC}$ induced by $\cF$ upon 
clean intersection. The principal $\mathrm{SO}(n)_\kappa$-action takes these horizontal curves at $C$ to
 horizontal curves in $\Theta_{\cC}^{\kappa}$ based at any matrix in the fiber.


Conversely, let $C=B\Lambda$ and assume that the horizontal lifts at $CB'$, $B'\in \mathrm{SO}(n)_{\Lambda}$, are contained
in $\cC$. Then they are inside the corresponding leaf of $\cF$ and therefore   
\[H\phi(\Omega)\subset \bigcap_{B'\in \mathrm{SO}(n)_{\Lambda}} \mathrm{Ad}_{BB'}(\mathcal{D})=
\mathrm{Ad}_B\left(\bigcap_{B'\in \mathrm{SO}(n)_{\Lambda}} \mathrm{Ad}_{B'}(\cD)\right).\]
Because the exponential intertwines the adjoint action the latter intersection can be understood in $\ss$. If $\Lambda\in \mathfrak{d}^{\kappa}$
we have
\[\mathfrak{d}^{\kappa}=\bigcap_{B'\in \mathrm{SO}(n)_{\Lambda}} \mathrm{Ad}_{B'}(\mathfrak{d}),\quad \mathrm{SO}(n)_{\Lambda}=\mathrm{SO}(n)_\kappa.\]
We used property $\mathcal{CK}$ with respect to an arbitrary point $x\in \Omega$. If we select the point whose image lies in the stratum of smallest dimension, 
then we conclude that 
$H\phi(\Omega)$ cannot leave that stratum.

If $H\phi\subset \Theta^{\kappa}_\cP$, then at any point $x\in \Omega$ all orthonormal frames in $\mathrm{SO}(n)_{\kappa}$ are parallel. This 
means that the common eigendirections $\mathrm{Ad}_B(\mathfrak{d}^{\kappa})$ are parallel. Therefore the restriction of $H\phi$ to the 
foliation given by the parallel translates of $\mathrm{Ad}_B(\mathfrak{d}^{\kappa})$ in $\Omega$ is flat. Hence on each such (affine) subspace
it is a quadratic form with equal eigenvalues. Therefore in the local splitting of $\phi$ along orthogonal characteristics we shall have
a multiple of the standard quadratic form along $\mathrm{Ad}_B(\mathfrak{d}^{\kappa})$.
%
\end{proof}

The following result is more general than Theorem \ref{thm:propertyI-generic} in the Introduction.
\begin{theorem}\label{thm:regular-strata} If A Hessian metric $H\phi$ on $\Omega$ has property  $\cC$ and 
$H\phi(\Omega)$ is contained in a stratum of $\Theta_\cP$, then it has property $\mathcal{CK}$. In particular $\phi$ has orthogonal characteristics.
 
\end{theorem}
\begin{proof}
Let $H\phi(\Omega)$ be contained in $\Theta^{\kappa}_{\cP}$. This stratum is foliated by $\mathrm{SO}(n)\Theta^{\kappa}_{\cD}$ and
we want to show that for each curve $\gamma(t)$ in $\Omega$ the derivative of $H\phi(\gamma(t))$ is tangent to this foliation. By property $\cC$ for each $t_0$
there exists an orthogonal frame $C\in \cC$ such that the horizontal curve $A(t)$ at $C$ has derivative at zero tangent to $\cC$:
\[A'(0)\in T\cC\cap T\pi^{-1}(\Theta^{\kappa}_{\cP})=C\cdot \mathrm{ad}_\Lambda^{-1}(\mathfrak{d})\cap C\cdot (\mathfrak{so}(n)+\mathfrak{d}^{\kappa})=
C\cdot (\mathfrak{so}(n)_\Lambda+\mathfrak{d}^{\kappa})=T\pi|_{\cC}^{-1}(\Theta_\cP^{\kappa}).
\]
Therefore the intersection $\cC\cap \pi|_{\cC}^{-1}(\Theta_\cP^{\kappa})=\Theta^{\kappa}_{\cC}$ is clean
and $A'(0)$ belongs to $\Theta^{\kappa}_{\cC}$. Because the vector is horizontal by item (2) in Proposition \ref{pro:desingularisation2} it is 
tangent to the foliation  $\mathrm{SO}(n)\Theta^{\kappa}_\cD$. Thus by item (3) the
tangent vector of $H\phi(\gamma)$ at $t_0$ is tangent to the foliation $\mathrm{SO}(n)\Theta^{\kappa}_{\cD}$. Because 
$\Omega$ is connected this implies that $H\phi(\Omega)$ is contained in one of the leaves of $\mathrm{SO}(n)\Theta^{\kappa}_{\cD}$. 
\end{proof}

The following result is a more precise statement that Theorem \ref{thm:propertyI-analytic} in the Introduction:  

\begin{theorem}\label{thm:Cartan-Kahler-analytic} Let $H\phi$ be a real analytic Hessian metric on $\Omega\subset \R^{n}$.
Then $\phi$ has property    $\cC$ if and only if it has property $\mathcal{CK}$.
In such case $H\phi$ is the restriction to $\Omega$ of a product Hessian metric on a (rotated) cube.
\end{theorem}
\begin{proof}
The stratification $\Theta_\cP$ has a finite number of strata. Therefore there exists one stratum 
$\Theta^{\kappa}_\cP$ whose pullback by $H\phi$ has non-empty interior $\Omega'\subset \Omega$.
 By Theorem \ref{thm:regular-strata} and Theorem \ref{thm:Cartan-Kahler} $H\phi(\Omega')\subset \mathrm{Ad}_B({\Theta^{\kappa}_\cD})$, $B\in \mathrm{SO}(n)$.
 In particular $H\phi(\Omega)$ must be contained in the real analytic submanifold $\mathrm{Ad}_B(\cD)$\footnote{The real analytic closure of 
  $\mathrm{Ad}_B(\Theta^{\kappa}_\cD)$ is the exponential of $\mathrm{Ad}_B(\mathfrak{d}^\kappa)$. This means that $H\phi(\Omega)$
  can only intersect strata of dimension equal or less to that of $\Theta_\cP^{\kappa}$; the equidimensional strata are those which under permutations go to 
  open subsets of $\mathfrak{d}^{\kappa}$).}. By Proposition 
  \ref{pro:desingularisation2} the restriction 
  \[\pi|_{\cC}:B\cD\to \mathrm{Ad}_B({\cD})\]
 is a diffeomorphism from a horizontal submanifold. Therefore all lifts of curves in $\Omega$ at $C\in \Theta^{\kappa}_{\cC}$ 
  are contained in $B\cD\subset \cC$, which is property $\mathcal{CK}$.
 
The image of $\Omega$ by the orthogonal projection onto a characteristic line is connected, and hence an interval.
The restriction of $\phi$ to the foliation of $\Omega$ by affine lines parallel to the characteristic line is locally projectable.
Because the interval has trivial topology local projections must agree on overlaps. 
\end{proof}

\section{Bifurcation of orthogonal characteristics along quadratic forms}\label{sec:bifurcations}

We shall construct a family of Hessian metrics with property $\cI$ which do not have orthogonal characteristics, and, hence, by Theorem 
\ref{thm:Cartan-Kahler} do not have property $\mathcal{CK}$. The family, despite being defined on domains which are not necessarily convex, 
will be also invariant under Legendre transform.

Firstly, we shall describe the domains $\Omega$ we are interested in.
Let $\Omega^{0}$ be an (open convex) \emph{polytope}. 
By this we mean a domain defined as the points where a finite number of affine maps are strictly positive; the zero set of each such map is a 
\emph{supporting hyperplane}. The closure of the polytope need not be compact.
Let $\R^{1}_l\times \R^{n-1}_l$ be the result of applying an orthogonal transformation $B_l$, $1\leq l\leq m$, to the splitting
$\R^{n}=\R\times \R^{n-1}$. We consider the polytope $\Omega^{1}_l=I_l\times F_l$, where the factors are polytopes in $\R_l^{1},\R^{n-1}_l$, respectively;
we shall refer to $\R^{1}_l$ as the \emph{primary characteristic} of $\Omega^{1}_l$.
We shall assume that $\R_l^{1}$ is oriented and we shall denote by $p_l$ the infimum of the interval $I_l$ (the interval may not be bounded from above).
we shall refer to $H_l:=p_l\times \R^{n-1}_l$ as the \emph{primary supporting hyperplane}.

 We shall assume that
\begin{itemize}
 \item the polytopes $\Omega^{0},\Omega^{1}_1,\dots,\Omega^{1}_m$ are disjoint 
 and that $\Omega^{1}_i$ and $\Omega^{1}_j$, $i\neq j$, have disjoint closure;
\item the primary supporting hyperplane $H_l$ for  $\Omega^{1}_l$ is also a supporting 
hyperplane for $\Omega^{0}$ and $\partial \Omega^{1}_l\cap H_l\subset \partial \Omega^{0}\cap H_l$.
\end{itemize}
We define $\Omega$ to be 
\begin{equation}\label{eq:polytope-with-handles}
\Omega=\Omega^{0}\bigcup (\Omega^{1}_1\cup p_1\times F_1)\bigcup \cdots \bigcup (\Omega^{1}_m\cup p_m\times F_m).
\end{equation}
We refer to $\Omega$ as in (\ref{eq:polytope-with-handles}) as a {\bf polytope with 1-handles}.

Secondly, we shall introduce appropriate strictly convex functions on the polytope with 1-handles.
Let $\phi_{0}$ be a multiple of the standard quadratic form on $\R^{n}$:
\[\phi_{0}(x)=k(x_1^{2}+\cdots +x_n^{2}).\]
Let $y=(y_1,\dots,y_n)$ be the coordinates which correspond to the image by $B_l$ of the canonical basis $e_1,\dots,e_n$ and let 
$q_l(y)=q_l(y_2,\dots,y_l)=k(y_1^{2}+\cdots +y_n^{2})$. Then we have:
\begin{equation}\label{eq:quadratic-orthogonal}
\phi_0(y)=k(y_1^{2}+\cdots +y_n^{2})=ky_1^{2}+q_l(y_2,\dots,y_l).
\end{equation}
%


\begin{proposition}\label{pro:bifurcations} Let $\phi$ be the function defined on the polytope with 1-handles $\Omega$ as follows:
 \[\phi|_{\Omega^{0}\cup p_1\times F_1\cup\cdots \cup p_m\times F_m}=\phi_0,\quad \phi|_{\Omega^{1}_l}=\phi_{l}+q_l,\]
 where $\phi_l(y)=\phi_l(y_1)$ a strictly convex smooth function on $I_l$ tangent at $p_l$ to $ky_1^{2}$ at infinite order. 
Then it has the following properties:
\begin{enumerate}
 \item It is a smooth and strictly convex function on $\Omega$.
 \item It has property $\cI$.
 \item The image $H\phi(\Omega)\subset \cP$ is contained in the union of the closed stratum and the two open strata of lowest dimension of $\Theta_\cP$.
  \item If there are two 1-handles on which $\phi_l$ is not a quadratic form and the primary characteristic are neither equal not perpendicular,
  then $H\phi$  does not have property $\mathcal{CK}$.
\end{enumerate} 
\end{proposition}
\begin{proof}
That $\phi$ is smooth and strictly  convex is a consequence of (\ref{eq:quadratic-orthogonal}) and of the definition of $\phi_l$. 

The restriction of $\phi$ to $\Omega^{0}$ has property $\cI$; the restriction to each $\Omega^{1}_l$ has orthogonal characteristics given by the columns of 
$B_l$. Therefore $\phi$ has property $\cI$ on the closure of the union, which is $\Omega$.

 By construction $H\phi$ sends $\Omega^{0}$ to the closed stratum of $\Theta_\cP$, and thus so $\overline{\Omega^{0}}$; if $\phi_l$ is not a quadratic form, then 
 $H\phi$ sends an open subset of $\Omega^{1}_l$ to the strata positive matrices with two eigenvalues so that one is simple:
\[H\phi(\overline{\Omega^{0}})\subset \bigcap_{B\in \mathrm{SO}(n)}\mathrm{Ad}_B(\cD),\quad H\phi(\Omega^{1}_l)\cap
  \mathrm{Ad}_{B_l}\left(\bigcap_{B'\in S(\mathrm{O}(1)\times \mathrm{O}(n-1))}\mathrm{Ad}_{B'}(\cD)\right)\neq \emptyset.\]

By Theorem \ref{thm:Cartan-Kahler} if $H\phi$ has property $\mathcal{CK}$ then $\phi$ has orthogonal characteristics
for some $C\in \cC$ (or $B\in \mathrm{SO}(n)$). On a 
 1-handle $\Omega_l^{1}$ with $\phi_l$  different from a quadratic form $\phi$ has orthogonal characteristics exactly for all $B_l\mathrm{SO}(n)_\kappa \cD$. 
The 1-handles $\Omega_i^{1}$ and $\Omega_j^{1}$ have primary characteristics which are neither equal nor orthogonal 
if and only if $B_i\mathrm{SO}(n)_\kappa\cap  B_j\mathrm{SO}(n)_\kappa=\emptyset$.
\end{proof}

\begin{proposition}\label{pro:Legendre-polytope-1-handles} Let $\Omega$ be a polytope with 1-handles and let $\phi\in C^{\infty}(\Omega)$ as in Proposition \ref{pro:bifurcations}.
Then the following holds:
\begin{enumerate}
 \item The Legendre map $d\phi$ on $\Omega$ is a diffeomorphism, its image $\Omega^{*}$ is a polytope with 1-handles, and ${\Omega^{1}_l}^{*}$ and $\Omega^{1}_l$ have the same primary characteristic.
 \item $\phi^{*}$ is a function as in Proposition \ref{pro:bifurcations}.
\end{enumerate}
\end{proposition}
\begin{proof}
Because $\phi$ is strictly convex  $d\phi:\Omega \to \R^{n}$ is a local diffeomorphism. 
Because $\phi|_{\Omega^{0}}$ is a quadratic form
$d\phi(\Omega^{0})$ is another polytope. The restriction  $\phi|_{\Omega^{1}_l=I_l\times F_l}$ decomposes as 
a sum of strictly convex functions $\phi_l+q_l$. The subset $d\phi(\Omega^{1}_l)$ is another  1-handle
because $I_l$ is 1-dimensional, $F_l$ is a polytope and $q_l$ is a quadratic form;
furthermore, because the Legendre transform commutes with orthogonal transformations
 $d\phi(\Omega^{1}_l)=I_l^{*}\times F_l^{*}$ where the product decomposition is also with respect to $\R_l^{1}\times \R^{n-1}_l$. 
The condition on the non-overlap of the closures of the 1-handles can be restated as follows: if two 1-handles have common primary supporting hyperplane, 
then their polytopes there have non-intersecting closure. This implies that if we prolongue each $I_l\subset \R^{_1}_l$ across $p_l$ to a
larger interval $\tilde{I}_l$ so that $\tilde{I}_l\times F_l\subset \Omega^{0}\cup p_l\times F_l\cup \Omega^{1}_l$, then 
\[d\phi(\Omega^{0})\cap d\phi(p_l\times F_l\cup \Omega^{1}_l)=\emptyset,\quad  d\phi(d\phi(\tilde{I}_i\times F_i)\cap d\phi(\tilde{I}_j\times F_j)\cap (\R^n\backslash d\phi(\Omega^{0}))=\emptyset,\quad i\neq j.\]
Therefore $d\phi:\Omega\to d\Omega$ is a bijection and thus a diffeomorphism onto another polytope with 1-handles.

 
Because the Legendre transform of a multiple of the standard quadratic form is a multiple of the standard quadratic form 
it follows that $\phi^{*}$ belongs to the class of functions described in  Proposition \ref{pro:bifurcations}.
\end{proof}

\section{An application to Poisson geometry}\label{sec:Poisson-commuting}
We shall show that property $\cI$ for strictly convex functions  condition is equivalent to the Poisson commuting equation for Poisson structures related to Kahler forms on toric
varieties. We shall use 
\begin{enumerate}[(a)]
 \item our classification of real analytic inversible Hessian metrics to deduce a factorization result;
\item the family of strictly convex functions with property $\cI$ introduced in  Section \ref{sec:bifurcations} to produce pencils of Poisson structures on regions of projective varieties which interpolate
from a Kahler structure to a Poisson structure with a finite number of Kahler leaves.
\end{enumerate}

\begin{definition}\label{def:toric-Poisson}\cite[Section 4]{BFM} A Poisson structure $\Pi$ on a toric variety $(X,\T)$ is 
\begin{enumerate} 
\item {\bf toric} if the bivector field $\Pi$
is $\T$-invariant, of type $(1,1)$ and positive, and if the symplectic leaves of  $\Pi$ equal the finitely many orbits of the torus action;
\item{\bf totally real} if the orbits of the (maximal) compact torus $T$ are coisotropic submanifolds. 
\end{enumerate}
 \end{definition}
\begin{remark}\label{rem:dformation} Totally real toric Poisson structures are good candidates to be limits of Hamiltonian Kahler forms: 
for such a form its 
%
inverse Poisson bivector
is $T$-invariant, of type $(1,1)$ and positive; there is a unique symplectic leaf of which the $T$-orbits are Lagrangian submanifolds.
Thus it is natural to look for converging sequences of such bivectors so that in the limit
the unique symplectic leaf breaks into the finitely many orbits and the $T$-symmetry is enlarged to $\T$-symmetry. One possible source 
would be a totally real toric Poisson structure which Poisson commutes with (the inverse of) a Hamiltonian Kahler form. In such case the convex combination
of the bivectors would be a smooth family of (inverses of) Hamiltonian Kahler forms converging to the totally real toric Poisson structure.
\end{remark}

On a toric variety a $\T$-invariant Poisson structure has a simple infinitesimal description: its restriction to the open dense orbit ---
which upon fixing a point  is identified
with $\T$ --- followed by the logarithm map,  defines a constant Poisson structure in the Lie algebra of $\T$.
The infinitesimal counter part of a toric Poisson structure is a Hermitian inner product. If it is totally real it corresponds to an inner
product on $i\tt$, where $\tt$ denotes the Lie algebra of $T$. In such case, we say that $e_1,\dots,e_n\in i\tt$ is an \emph{adapted Darboux basis}
if the inner product becomes standard; equivalently, $e_1,\dots,e_n,ie_1,\dots,ie_n$ is a Darboux basis for the inverse constant symplectic 
structure. 

On a toric variety a Hamiltonian Kahler form  can be described by (Legendre dual) strictly convex functions:
a \emph{Kahler potential} in logarithmic coordinates and a \emph{symplectic potential} in momentum map coordinates \cite{Gui}.

The Poisson commuting equation for a totally real Poisson structure and a Hamiltonian Kahler form corresponds --- in  appropriate coordinates ---
to property $\mathcal{I}$:

\begin{theorem}\label{thm:inversible-commuting} Let $(X,\T)$ be a toric variety endowed with a totally real toric Poisson structure $\Pi$
and a Kahler structure $\sigma$ for which 
the  action of $T$ is Hamiltonian. Let $P$ be the inverse Poisson structure of $\sigma$. Then the following statements are equivalent:
\begin{enumerate}
 \item $\Pi$ and $P$ Poisson commute: $[\Pi,P]=0$.
 \item In an adapted Darboux basis the Kahler potential $\phi$ 
 has property $\cI$.
 \item In an adapted Darboux basis the symplectic potential $\phi^{*}$
has property $\cI$. 
\end{enumerate}    
\end{theorem}

\begin{proof}
We regard the equation $[\Pi,P]=0$ as the defining equation for degree 2-cocycles in the Poisson cohomology of $\Pi$: $d_\Pi P=0$.
In logarithmic coordinates $\exp^{*}\Pi$ has an inverse which is a (constant) symplectic structure $\Xi$ on  $\tt\oplus i\tt$.
Therefore 
\[\Xi^\#:(\mathfrak{X}^\bullet,d_{\exp^{*}\Pi})\to (\Omega^\bullet, d)\]
is an isomorphism of chain complexes (see e.g. \cite[Proposition 6.12]{GHP}). Hence 
\[ [\exp^{*}\Pi,\exp^{*}P]=0 \Longleftrightarrow d\omega=0,\quad  \omega(X,Y):=\exp^{*}P(\Xi^\# X ,\Xi^\# Y).\]

Let $e_1,\dots,e_n,ie_1,\dots,ie_n$ be an adapted Darboux basis for the totally real toric Poisson structure. 
In the fixed coordinates and associated frames of the complexified tangent and cotangent bundles 
the matrices of $\Pi$ and $\Xi$ are:
\[\Pi^\#=\frac{2}{i}\begin{pmatrix} 0 & \mathrm{I}_n \\ -\mathrm{I}_n & 0         
        \end{pmatrix},\quad \Xi^\#=\frac{i}{2}\begin{pmatrix} 0 & -\mathrm{I}_n \\ \mathrm{I}_n & 0         
        \end{pmatrix}
\] 
If we let  \[\frac{2}{i}\begin{pmatrix} 0 & g \\ -g & 0         
        \end{pmatrix}\] denote the matrix of $\exp^{*}P^\#$, 
then
\[\w^\#=-\Xi^\# \exp^{*}P^\# \Xi^\#=-\frac{i}{2}\begin{pmatrix} 0 & -\mathrm{I}_n \\ \mathrm{I}_n & 0         
        \end{pmatrix}\frac{2}{i}\begin{pmatrix} 0 & g \\ -g & 0         
        \end{pmatrix}\frac{i}{2}\begin{pmatrix} 0 & -\mathrm{I}_n \\ \mathrm{I}_n & 0         
        \end{pmatrix}=\frac{i}{2}\begin{pmatrix} 0 & g \\ -g & 0         
        \end{pmatrix}\] 
Hence 
\[\w=\w(x)=\frac{-i}{2}\sum_{j,k}g_{jk}(x)dz_j\wedge d\bar{z}_k,
\]
Its exterior derivative is:
\[\begin{split}
d\w &= -\frac{i}{2}\sum_{i,j=1}^n dg_{ij}\wedge dz_i\wedge d\bar{z}_j=
-\frac{i}{2}\sum_{i,j=1}^n \sum_{k=1}^n \left(\frac{\partial g_{ij}}{\partial z_k}dz_k+
\frac{\partial g_{ij}}{\partial \bar{z}_k}d\bar{z}_k\right)\wedge dz_i\wedge d\bar{z}_j= \\
&= -\frac{i}{2}\left(\sum_{i,j,k=1}^n  \left(\frac{\partial g_{jk}}{\partial z_i}-
\frac{\partial g_{ik}}{\partial z_j}\right)\wedge dz_i\wedge dz_j\wedge  d\bar{z}_k+
\sum_{i,j,k=1}^n  \left(\frac{\partial g_{ij}}{\partial \bar{z}_k}-
\frac{\partial g_{ik}}{\partial \bar{z}_j}\right)\wedge dz_i\wedge d\bar{z}_j\wedge  d\bar{z}_k\right)\\
\end{split} 
\]
Because the entries of $g$ only depend on $x$ we have:
\[\frac{\partial g_{jk}}{\partial z_i}-
\frac{\partial g_{ik}}{\partial z_j}=\frac{1}{2}\left(\frac{\partial g_{jk}}{\partial x_i}-
\frac{\partial g_{ik}}{\partial x_j}\right),\quad 
 \frac{\partial g_{ij}}{\partial \bar{z}_k}-
\frac{\partial g_{ik}}{\partial \bar{z}_j}=\frac{1}{2} \left(\frac{\partial g_{ij}}{\partial x_k}-
\frac{\partial g_{ik}}{\partial x_j}\right)
\]
The matrix of $\exp^{*}\sigma$ is \[\frac{i}{2}\begin{pmatrix} 0 & g^{-1} \\ -g^{-1} & 0         
        \end{pmatrix},\]
and $g^{-1}$ is the Hessian of the Kahler potential $\phi$. In particular $g$ and its inverse are symmetric matrices.         
Renaming the set of indices in the first  summand and using the symmetry of $g$ we obtain:
\[
d\w=0\Longleftrightarrow \frac{\partial g_{ij}}{\partial x_k}-
\frac{\partial g_{ik}}{\partial x_j},\quad  1\leq  i,j,k \leq  n.
\]
This is exactly property $\cI$ for the Hessian metric $g^{-1}=H\phi$.

The symplectic potential of $\sigma$ is the Legendre transform of $\phi$. By Proposition \ref{pro:legendre}
$\phi$ (in $\R^{n}$) has property $\cI$ if and only if $\phi^{*}$ (in $d\phi(\R^{n})$) has property $\cI$.
\end{proof}

\begin{theorem}\label{thm:no-analytic} Let $(X,\T)$ be a projective toric Poisson structure endowed with a
 toric Poisson structure $\Pi$ which Poisson commutes with a real analytic Kahler structure $\sigma$ for which 
the action of $T$ is Hamiltonian. Then  $(X,\T)$ is a Cartesian product of projective lines and both 
$\Pi$ and $\sigma$ factorize.
\end{theorem}
\begin{proof}
Because $\sigma$ is real analytic the Kahler potential $\phi$ is real analytic; the Legendre transform preserves analytic (strictly convex) functions.
Therefore by Theorem \ref{thm:inversible-commuting} the symplectic potential $\phi^{*}$ 
has property $\cI$. By Theorem \ref{thm:Cartan-Kahler-analytic} $\phi^{*}$ is defined in a Cartesian product 
of intervals $I_1\times \cdots \times I_n$ (we may dispense with the rotation by changing accordingly the adapted Darboux basis).
One must have the equality $d\phi(\R^{n}) =I_1\times \cdots \times I_n$ because otherwise by repeating the Legendre transform 
we would get a domain for the original Kahler potential strictly containing $\R^{n}$. 
Thus the interior of the moment polytope $\Delta$ is a Cartesian product of intervals. A product of intervals
has a property invariant under affine transformations: it is limited by pairs of 
parallel hyperplanes. Because $\Delta$ is a Delzant polytope there is an affine transformation that takes the integral lattice of $\tt^{*}$ to 
$\Z^{n}$, a vertex of $\Delta$ to the origin and the facets containing this vertex to the coordinate hyperplanes. Because $\Delta$ must be still 
described by parallel hyperplanes, is it actually a Cartesian product of intervals in this integral affine coordinates of $\tt^{*}$.
The fan of the Delzant polytope determines the toric variety $(X,\T)$. The fan of a cube in $(\R^{n},\Z^{n})$ corresponds 
to $\mathbb{C}P^{1}\times \cdots \times \mathbb{C}P^{1}$. 

To show that the Kahler form $\sigma$ also splits as a sum of Kahler forms on each projective line we use toric charts for $(X,\T)$.
For that we observe that the linear part of the above affine transformation must be a permutation followed by a (signed) re-scaling of each Euclidean direction.
Therefore if we dispense the affine transformation we deduce that the in the fixed compatible Darboux basis the subset  $ie_1,\dots,ie_n$ is --- up to re-scaling of its 
members --- an integral basis of $\tt$. Let us re-scale to a basis $\epsilon_1,\dots,\epsilon_n,i\epsilon_1,\cdots,i\epsilon_n$ 
so that the second block is an integral basis of $i\tt$. In the corresponding coordinates
the Kahler potential of $\sigma$ is still a sum of strictly convex functions on each coordinate and therefore  the
Legendre diffeomorphism still sends $i\tt\cong \R^{n}$ to a cube. 
Let $v$ be the vertex of its closure
whose coordinates are smaller than those of the others.
The basis of inner pointing integral vectors normal to the facets containing $v$ is exactly 
$i\epsilon_1,\cdots,i\epsilon_n$. Therefore for the standard toric chart  associated to $v$ (see e.g. \cite[Chapter2, Section 5]{A})
\[(\C^{*})^{n}\circlearrowright \C^{n}\subset \mathbb{C}P^{1}\times \cdots \times \mathbb{C}P^{1},\quad  (0,\dots0)=([0:1],\dots,[0:1]), \]
the identification\footnote{Though we do not need it here, the toric chart could be chosen compatible with the monoid structure, 
so that $(1,\dots,1)$ corresponds to the fixed point in the open orbit.} of $\T$ with $(\C^{*})^{n}$ comes from the Lie algebra identification  which  sends
$i\epsilon_1,\dots,i\epsilon_n$ to the canonical basis of $\R^{n}\subset \C^{n}$.
In other words, upon having identified $\T$ with the open orbit of $X$, the product structure induced by 
$X\cong  \mathbb{C}P^{1}\times \cdots \times \mathbb{C}P^{1}$ on $\T$ is exactly the factorisation of the torus coming from the (complex) basis
$e_1,\cdots,e_n$ of its Lie algebra. 
The factorisation of $X$ decomposes $\sigma=\sigma_1+\cdots +\sigma_n$, $\sigma_k\in \Omega^{2}(X)$.  Each  
 $\sigma_k$ is projectable by 
\[X\to \mathbb{C}P^{1}\times \cdots \widehat{\mathbb{C}P^{1}}\times \cdots \times \mathbb{C}P^{1}\]
because it has that property in the open dense subset $\T$. The compatibility of $\Pi$ with the product structure is also immediate. Therefore 
\[(X,\T,\Pi,\sigma)\cong (\mathbb{C}P^{1},\C^{*},\Pi_1,\sigma_1)\times \cdots \times 
 (\mathbb{C}P^{1},\C^{*},\Pi_n,\sigma_n),
\]
where symplectic forms are invariant under the action of $T$. 
\end{proof}

Let  $\Omega$ be a polytope with 1-handles. Its \emph{outer boundary} will be the subset of the boundary lying 
in supporting hyperplanes for the 1-handles which
are parallel to the primary ones but not equal to them; its \emph{inner boundary} will be the subset of the 
boundary in supporting hyperplanes of the polytope which are not primary supporting hyperplanes of some 1-handle. 

\begin{theorem}\label{thm:handles}
Let $(X,\T)$ be a toric variety  with fan given  by the  polytope $\Delta\subset (\R^{n},\Z^{n})$.
Let $\Omega$ be a polytope with 1-handles such that the intersection $\partial \Omega \cap \partial \Delta$ is contained 
in the union of the outer boundary of $\Omega$ and an orthogonal subset of supporting hyperplanes of the inner boundary of $\Omega$.

Then there exist an open subset $X_\Omega\subset X$ invariant under the action of $T$ and 
a Kahler form $\sigma\in \Omega^{2}(X_\Omega)$ with the following properties:
\begin{enumerate}
 \item The action of $T$ on $(X_\Omega,\sigma)$ is Hamiltonian with momentum map 
  the union of $\Omega$ with the interior on each face of $\Delta$ of the outer and inner boundaries of $\Omega$.
 \item The Poisson structure with corresponds to $\sigma$ Poisson commutes with any toric Poisson structure
 for which the canonical basis of $\R^{n}$ is an adapted 
 Darboux basis up to scaling.
\end{enumerate}
\end{theorem}
\begin{proof}
We start with  $\phi\in C^{\infty}(\Omega)$ a function as in Proposition \ref{pro:bifurcations} on which  we shall impose natural boundary conditions
\footnote{The functions will satisfy well-known boundary conditions to produce Kahler 
metrics (see e.g. \cite[Chapter 2]{A}). These metrics/complex structures are constructed fixing the symplectic 
structure. Because we are interested in keeping fixed the complex structure, we are going to be very explicit with the computation of the Kahler potentials
and corresponding symplectic forms.}
Let
$\Omega^{1}_l$ be a 1-handle whose supporting hyperplane in the outer boundary of $\Omega$ intersects $\partial \Delta$. Let
$\alpha_l$ be the unique integral affine map which vanishes in the supporting hyperplane and it is positive on $\Delta$. 
We shall assume that $\phi_l$ equals $\frac{1}{2}\alpha_l(\log(\alpha_l)-1)$ near the end of $I_l$ opposite to $p_l$. This 
is always possible because the existing constraint on $\phi_l$ is near $p_l$.

The region $X_\Omega\subset X$ is the result of adding certain points to  $\exp(\Omega^{*}\oplus i\R^{n})\subset X$. By item (1) in Proposition 
\ref{pro:Legendre-polytope-1-handles} the primary characteristics of ${\Omega^{1}_l}^{*}$ and $\Omega^{1}_l$ are the same and
the orientation is also preserved. Because $\Delta$ determines a fan the primary characteristic $\R^{1}_l\subset \R^{n}\cong i\tt$ determines
a 1-parameter subgroup  $\C_l^{*}\subset \T$ together with an isomorphism $\C^{*}_l\cong \C^{*}$ (the Lie algebra is trivialized by a positive 
integral vector in $\R^{1}_l$). Because the derivative of $\phi_l$ near the boundary point opposite to $p_l$ goes to infinity $I^{*}_l\subset \R^{1}_l$
is a semi-infinite interval in the positive half line; in  particular it is a semigroup; let $\mathbb{D}^{o}_l\subset \C_l^{*}$
be the semigroup $\exp(I^{*}_l\oplus i\R^{1}_l)$. This semigroup acts (freely) on $\exp({\Omega^{1}_l}^{*}\oplus i\R^{n})$. 
If we let $\T_l\subset \T$ be the subtorus
which exponentiates the complexification of $\R^{n-1}_l$, then we can factor
\begin{equation}\label{eq:product-structure}
\exp({\Omega^{1}_l}^{*}\oplus i\R^{n})=\mathbb{D}^{o}_l\times \exp(F_l\oplus i\R^{n-1}_l)\subset \T=\C^{*}_l\times \T_l.
\end{equation}
We define $X_\Omega\subset X$ to be the union of $\exp(\Omega^{*}\oplus i\R^{n})$ with the closure of every  $\mathbb{D}^{o}_l$-orbit
in $\exp({\Omega^{1}_l}^{*}\oplus i\R^{n})$,
for every 1-handle  whose supporting hyperplane in the outer boundary of $\Omega$ intersects $\partial \Delta$ (for the moment we assume that no inner
boundary components are in $\partial \Delta$).

The function $\phi^{*}$ defines a Kahler form $\sigma=i\partial\bar{\partial}\log(\phi^{*})$ on $\exp(\Omega^{*}\oplus i\R^{n})$ for which the action of 
$T$ is Hamiltonian. We want to 
argue that $\sigma$ extends to a Kahler form on $X_\Omega$. 

Firstly, we show how upon adding the orbit closures to $\exp({\Omega^{1}_l}^{*}\oplus i\R^{n})$ (and not in 
$\exp(\Omega^{*}\oplus i\R^{n})$)
we get an  open subset $X_l\subset X$ which is $T$-invariant and 
to which the product structure in (\ref{eq:product-structure}) extends. 
For that we use the toric atlas (as a monoid) determined by the polytope $\Delta$ (\cite[Chapter2, Section 5]{A}):
To each vertex $v\in \Delta$ there correspond a toric chart which identifies the union of orbits of $X$ which correspond to 
the star of $v$ with the standard affine toric variety: $(\mathbb{C}^{n},{(\C^{*})}^{n})$.
The standard integral basis of the Lie algebra of ${(\mathbb{S}^{1})}^{n}$ comes from the integral linear forms $\nu_{l_1},\dots,\nu_{l_n}$ associated to 
the affine forms $\alpha_{l_j}$; in particular this describes how for each toric chart $\tt\oplus i\tt$ --- for which we already have picked a basis ---
is identified with the Lie algebra 
of the standard complex torus $\R^{n}\oplus i\R^{n}$.
The point in the open orbit of $X$ which determines the monoid structure goes to the unit $(1,\dots,1)\in \C^{n}$.

Let us fix a toric chart of a vertex $v$ which belongs to the supporting hyperplane of $\Omega^{1}_l$
and $\Delta$. Under the identification of $\T$ with the standard torus ${(\C^{*})}^{n}$ we can assume that the (trivialized) subgroup  $\C_l^{*}\cong \C^{*}$ 
maps to the first factor of the standard torus so that on trivializations the isomorphism is given by the inversion. Thus 
$\mathbb{D}^{o}_l$ maps to a semigroup $\mathbb{D}^{i}_l\subset \C^{*}$ contained in the unit disk. Let $W_l\subset \C^{n}$
be the image in the toric chart of the second factor  $\exp(F_l\oplus i\R^{n-1}_l)$. Then the image of $\exp({\Omega^{1}_l}^{*}\oplus i\R^{n})$
is 
\[(zw_1,\dots,w_n),\quad z\in \mathbb{D}^{o}_l,\quad w=(w_1,\dots,w_n)\in W_l.\]
Therefore the image of $X_l$ is also completely contained in the toric chart and  equals
\[(zw_1,\dots,w_n),\quad z\in \mathbb{D}^{o}_l\cup \{0\},\quad (w_1,\dots,w_n)\in W_l.\]
Because $W_l$ is a codimension 2 submanifold
which intersects each complex line parallel to the $z_1$-axis transversely in at most one point,
we deduce that $X_l$ is an open subset which extends the product structure. By construction it is also $T$-invariant. 

To show that $\sigma$ extends to a Kahler form on $X_l$ we shall work with its inverse Poisson structure $P$. 
The decomposition  $\phi^{*}|_{{\Omega^{1}_l}^{*}}(y)=\phi^{*}_l(y_1)+q^{*}_l(y_2,\dots,y_1)$ 
implies that on the exponentiation of the 1-handle  $P$ decomposes as $P_l+P'_l$, where each summand  is a field of bivectors tangent
to one of the foliations
in the product decomposition (\ref{eq:product-structure}). The second field of bivectors is easier to describe:
In the Lie algebra the foliation is given by translates of $F_l\subset \R^{n-1}_l$. On each such leaf the Kahler potential for the corresponding
Kahler form is the quadratic form $q_l$. Therefore $P_l'$ corresponds to a constant bivector on ${\Omega^{1}}^{*}\times i\R^{n}$.
The exponentiation of a constant bivector to  the (abelian) Lie group has an alternative description: it is the field of bivectors obtained by replacing
each vector in the decomposition in $\wedge^{2}(\tt\oplus i\tt)$ by its corresponding infinitesimal vector field for the action by (left) multiplication.
Because the action of $\T$ on itself extends to an action on $X$ it follows that $P_l'$ is the restriction of a (Poisson) structure on $X$.
To describe $P_l$ on each semigroup orbit we may assume that $\phi_l(y_1)$ equals $-\frac{1}{2}\alpha_l(\log(\alpha_l-1)$ everywhere in $I_l$.
The Legendre transform of $-\frac{1}{2}(-r)(\log(-r)-1)$ for $r<0$ is $\frac{1}{2}\exp^{-2r}$. Using that the Legendre transform commutes
with orthogonal transformations and its behavior under translation and scaling we obtain $\phi_l^{*}(y_1)=\frac{1}{2}\exp^{-2y_1/|\nu_l|}-y_1d_l$,
where $d_l$ is the distance of the boundary point of $I_l$ different from $p_l$. Since we are interested in Kahler forms/bivectors we may dispense with the linear summand.
Under the semigroup identification $\log:\mathbb{D}_l^{o}\to I_l^{*}\oplus i\R$ the potential pull backs to  $\frac{1}{2 z\bar{z}}$. 
Under the identification $\mathbb{D}_l^{o}\to \mathbb{D}_l^{i}$ it maps to $\frac{1}{2}z\bar{z}$. Under the action on the standard toric chart 
it maps to $\frac{1}{2a_1^{2}}z_1\bar{z}_1$. Hence the bivector $P_1$ there equals 
$\frac{2a_1^{2}}{i}\frac{\partial}{\partial z_1}\wedge \frac{\partial}{\partial \bar{z}_1}$ near $z_1=0$, which extends to $X_l$. Both $P_l$ and $P_l'$ are
non-degenerate in the added points, and thus $\sigma$ extends to a Kahler form. 

For boundary components in the inner boundary we change coordinates by a rotation so that all supporting hyperplane involved
are coordinate hyperplanes. Then we impose the same boundary conditions as above on the corresponding summands of the multiple of the standard quadratic form in these
coordinates. We may have supporting hyperplanes with non-empty intersection, say $k$ of them. This means that we shall have work with the corresponding coordinates and 
hence with a splitting into a vector subspace of dimension $k$ and its orthogonal complement. We shall work on a toric chart associated to a vertex 
in the intersection of the supporting hyperplanes. There, the foliation corresponding to the vector subspace will have leaves given by the action 
of ${(\C^{*})}^{k}$ on the first $k$ coordinates on an appropriate slice. Hence by adding closure of (semigroup) orbits we shall obtain an open 
subset. The computation of the Kahler potential for the inverse symplectic form of $P_l$ on such leaves is analogous.

To computation of the image of the momentum map is straightforward.

\end{proof}

The proof of Proposition \ref{pro:cp2} in the Introduction is a minor variation of the following:

\begin{example} (Attaching a toric 1-handle to the standard commuting pair)
Let $\Delta$ be the standard n-simplex in $\R^{n}$. Let $\Omega^{0}$ be the truncation of the (open) cube of side $(0,\tfrac{1}{n})$ by the hyperplane
$x_1+\cdots +x_n=\frac{2}{3}$. Let $\Omega^{1}$ be the 1-handle with the primary characteristic spanned by $(1,\cdots,1)$, 
and so that its primary supporting hyperplane is $x_1+\cdots +x_n=\frac{2}{3}$, the parallel one is $x_1+\cdots +x_n=1$ and the 
$n-1$-dimensional polytope is the (translation of) the intersection of the cube and the primary supporting hyperplane. 

We let $\Omega$ be the polytope with 1-handles determined by $\Omega^{0}$ and $\Omega^{1}$ above. It is in the hypotheses 
of Theorem \ref{thm:handles}. By going through the its proof we check that:
\begin{itemize}
 \item In the toric chart associated to the origin $\sigma$ (near the origin)
will be the standard (constant) Kahler form $\frac{i}{2}\sum_j dz_j\wedge d\bar{z}_j$; the open cube is a (punctured) polydisk which is appropriately truncated.
\item  Near the truncation 
hypersurface $W$ the Kahler form is  $\frac{i}{2}\sum_j \frac{1}{z_j\bar{z}_j} dz_j\wedge d\bar{z}_j$.
\item Attaching the 1-handle amounts to the following: the truncating hypersurface $W$ is stable under the diagonal action of $\mathbb{S}^{1}$.
Then each such orbit is being `capped' by a (holomorphic) disk which is Kahler for $\sigma$; the disk is nothing 
but (a part of) the projective line determined by the orbit, its center being in the hyperplane at infinity 
$\mathbb{C}P^{n}=\C^{n}\cup \mathbb{C}P^{n-1}$; this is done for the whole 
$F_1\times {(\mathbb{S}^{1})}^{n-1}$-family.
\item  The inverse of Kahler form $\sigma$ is a Poisson bivector field $P$ which Poisson commutes with the totally real toric Poisson  bivector field $\Pi$  which in the 
previous toric chart is $\frac{2}{i}\sum_j z_j\bar{z}_j \frac{\partial }{\partial z_j}\wedge \frac{\partial}{\partial \bar{z}_j}$. 
\end{itemize}

\end{example}

%
%


\begin{thebibliography}{}

 \bibitem{A} Apostolov, V. The Kahler geometry of Kahler manifolds. \url{http://www.cirget.uqam.ca/~apostolo/papers/toric-lecture-notes.pdf}.
\bibitem{BFM} Brambila, L; Frejlich, P.; Martínez Torres, D. Coregular submanifolds and Poisson submersions. Preprint arXiv:2010.09058. 
\bibitem{Du} Duistermaat, J. J. On Hessian Riemannian structures. Asian J. Math. 5 (2001), no. 1, 79--91.
\bibitem{DK} Duistermaat, J. J.; Kolk, J. A. C. Lie groups. Universitext. Springer-Verlag, Berlin, 2000.
\bibitem{GHP} Laurent-Gengoux, C.; Pichereau, A.; Vanhaecke, P. Poisson structures. Grundlehren der Mathematischen Wissenschaften, 347. Springer, Heidelberg, 2013.
\bibitem{Gui} Guillemin, V. Kaehler structures on toric varieties. J. Differential Geom. 40 (1994), no. 2, 285--309.
\end{thebibliography}
 \end{document}